\documentclass[reqno]{amsart}
\usepackage{amsmath, amsfonts, amssymb, amsthm, mathtools}

\usepackage{amscd}
\usepackage{bbm}
\usepackage{enumerate}
\usepackage{galois}
\usepackage{mathrsfs}
\usepackage{xypic}
\usepackage{geometry}
\usepackage{hyperref}
\usepackage{mathrsfs}
\usepackage{rotating}
\usepackage{cite}

\usepackage{color}

\theoremstyle{plain}
\newtheorem{thm}{Theorem}[section]
\newtheorem{cor}[thm]{Corollary}
\newtheorem{lem}[thm]{Lemma}
\newtheorem{prop}[thm]{Proposition}

\newtheorem{ex}{Example}[section]
\newtheorem{rem}{Remark}[section]

\def\pdf#1#2{\frac{\partial #1}{\partial #2}}

\def\tangentvector#1{\pdf{}{#1}}
\def\dif{\mathrm{d}}

\numberwithin{equation}{section}

\def\res{\mathop{\mathrm{Res\,}}}

\newcommand{\p}{\partial}

\newcommand{\al}{\alpha}

\newcommand{\dd}{\mathrm{d}}
\newcommand{\lm}{\lambda}

\newcommand{\bfx}{\mathbf{x}}

\begin{document}
	
	\title[Affine Weyl Groups and Generalized Frobenius Manifolds]{Generalized Frobenius Manifold Structures on the Orbit Spaces of Affine Weyl Groups II}
	
	\author{Lingrui Jiang, Si-qi Liu, Yingchao Tian, Youjin Zhang}
	
	\begin{abstract}
		This is a sequel to \cite{JTZ2025-1}, in which an approach to construct a class of generalized Frobenius manifold structures on the orbit spaces of affine Weyl groups is presented. In this paper we apply this construction to the affine Weyl groups of type $A_\ell, B_\ell, C_\ell$ and $D_\ell$.
	\end{abstract}

	\date{\today}
	
	\maketitle
	\tableofcontents
	
	\section{Introduction}\label{intro}
In \cite{JTZ2025-1}, we introduced an approach to construct a class of generalized Frobenius manifold structures on the orbit spaces of affine Weyl groups. As illustrative examples, we applied this construction to the affine Weyl groups of types $A_1, A_2, A_3, B_3, C_3, D_4$ and $G_2$. In the present paper, we extend this construction to the affine Weyl groups of types $A_\ell, B_\ell, C_\ell$ and $D_\ell$, thereby obtaining further examples of generalized Frobenius manifold structures. We begin by recalling the basic framework of this construction.

Let $R$ be an irreducible reduced root system in an $\ell$-dimensional Euclidean space with inner product $(\cdot\,,\cdot)$, and $\al_1,\dots,\al_\ell$ be a basis of simple roots. Denote by $\alpha_1^\vee,\dots, \alpha_\ell^\vee$ and $\omega_1,\dots,\omega_\ell$ the corresponding coroots and fundamental weights of $R$. We fix a weight 
\begin{align}
		\label{omega = sum mr omegar}
		\omega = \sum_{j=1}^{\ell} m_j \omega_j,\quad  m_j \in \mathbb Z_{\ge 0}
	\end{align}
of $R$, and introduce affine coordinates $(x^1,\dots,x^\ell; c)$ on $V$ by
	\begin{align}
		\label{x=c omega+}
		\bfx = c\omega + x^1\alpha_1^\vee + \dots + x^\ell\alpha_\ell^\vee,
	\end{align}
	where $c\in\mathbb{R}$ is a fixed parameter. Denote by $W(R)$ 
	and $W_a(R)$ respectively the Weyl group and the affine Weyl group associated with the root system $R$. The action of $W_a(R)$ on $V$ given by
\begin{align}
		\label{linear trans}
		\sigma\left(c\omega + x^1\alpha_1^\vee + \dots + x^\ell\alpha_\ell^\vee\right) =c\omega + \tilde x^1\alpha_1^\vee + \dots + \tilde x^\ell\alpha_\ell^\vee,\quad \sigma\in W_a(R)
	\end{align}	
yields the change of affine coordinates
\[\sigma(x^1,\dots,x^\ell; c)=(\tilde x^1,\dots, \tilde x^\ell; c),\quad \sigma\in W_a(R).\] 
This change of coordinates induces a right action on the Fourier function ring     
\[\mathscr F=\mathrm{span}_\mathbb{C}\left\{\mathrm{e}^{2\pi i(t_0c+t_1x^1+\dots+t_\ell x^\ell)}\mid
    t_0,t_1,\dots,t_\ell\in\mathbb{R}\right\},\]
which has a gradation defined by
\begin{align}
		\deg \mathrm{e}^{2\pi ix^j} = \theta_j,\quad  \deg \mathrm{e}^{2\pi i c} = -1,
	\end{align}
where $\theta_j=(\omega_j, \omega)$. It is shown in \cite{JTZ2025-1} that the $W_a(R)$ action on $\mathscr F$ preserves the degrees of monomials. We denote by $\mathscr F^W(R)$ the invariant subring of $\mathscr F$ \textit{w.r.t.} this action.

We introduce a parameter 
\begin{align}\label{zzh-2}
		\lambda = \mathrm{e}^{-2\pi i\kappa c}\ \textrm{with}\  \kappa = \gcd\{(\omega, \alpha_r)\mid r = 1,\dots,\ell\},
	\end{align}
and define the $\lambda$-Fourier polynomial ring $\mathscr A$ as a subring of $\mathscr F$ by
	\begin{align*}
		\mathscr A = \mathbb C[\lambda] \otimes \mathbb C\bigl[\mathrm{e}^{2\pi i x^j}, \mathrm{e}^{-2\pi i x^j}\mid j = 1,\dots,\ell\bigr] \subset\mathscr F.
	\end{align*}
Denote by $\mathscr A^W = \mathscr A \cap \mathscr F^W$ the $W_a(R)$-invariant $\lambda$-Fourier polynomial ring, then we know from \cite{JTZ2025-1}
that $\mathscr A^W = \mathbb C[y^1,\dots,y^\ell; \lambda]$,
where 
\begin{align}
		\label{yj definition}
		y^j(\bfx) = \mathrm{e}^{-2\pi i\theta_j c}Y_j(\bfx)
		= \frac 1{N_j} \mathrm{e}^{-2\pi i\theta_j c}\sum_{w\in W(R)} e^{2\pi i(\omega_j,  w(\bfx))}, 
		~j = 1,\dots,\ell,
	\end{align}
and $N_j = \#\{w\in W(R)\mid w(\omega_j) = \omega_j\}$.
The $W_a(R)$-invariant $\lambda$-Fourier polynomials $y^1,\dots, y^\ell$ are called basic generators of $\mathscr A^W$, and they are quasi-homogeneous of $\deg y^j=\theta_j$. 
 
In order to construct generalized Frobenius manifold structures on the orbit space of the affine Weyl group $W_a(R)$, we introduce in \cite{JTZ2025-1} the notion of proper generators of $\mathscr A^W$. 
Let $z^1,\dots,z^\ell \in \mathscr A^W$. 
We call $\{z^1,\dots,z^\ell\}$ a set of proper generators of $\mathscr A^W$ if $z^j\in \mathscr A^W$ with $\deg z^j = \theta_j$ for $j=1,\dots,\ell$, and
		\begin{align}
			\label{proper gen defn}
			z^j\big|_{\lambda=0}=y^j\big|_{\lambda=0},\quad j=1,\dots,\ell.
		\end{align}
In other words, the quasi-homogeneous $W_a(R)$-invariant $\lambda$-Fourier polynomials 
$z^1,\dots,z^\ell$ form a set of  proper generators of $\mathscr A^W$ if and only if for any $1\le j\le \ell$, there exists a polynomial $s^j$ of $\lambda$ and elements of $\{y^r\mid 1\le r\le\ell,\,\theta_r < \theta_j\}$, such that either $s^j=0$ or $\deg s^j = \theta_j - \kappa$, and 
		\begin{align}
			\label{proper gen cond}
			z^j = y^j + \lambda s^j.
		\end{align}
Note that the parameter $\lambda$ of the $W_a(R)$-invariant $\lambda$-Fourier polynomials $z^1,\dots,z^\ell$ is defined by \eqref{zzh-2} and it lies on the unit circle. However, since $z^j$ depend polynomially on $\lambda$, the definition of $z^j$ can be extended naturally to any $\lambda \in \mathbb C$, including $\lambda = 0$.

Given a set $\{z^1,\dots,z^\ell\}$ of proper generators of $\mathscr A^W$, we consider the orbit space 
\[\mathcal{M}:=\mathcal{M}(R,\omega)\cong\mathbb{C}^\ell\] 
of the affine Weyl group $W_a(R)$. For each parameter $\lambda\in\mathbb{C}$, the $\varphi_\lambda$-transformation 
\[\varphi_\lambda\colon (x^1,\dots,x^\ell)\mapsto (z^1,\dots,z^\ell)\] 
induces a pushing forward $(\varphi_\lambda)_*$,  which transforms the contravariant metric 
\begin{equation}\label{zh-07-17-5}
a(\dif x^i,\dif x^j)=a^{ij}\ \mathrm{with}\ (a^{ij})=\big((\alpha_i^\vee,\alpha_j^\vee)\big)^{-1},\quad i,j=1,\dots,\ell
\end{equation}
on $V\otimes\mathbb{C}$ to a contravariant metric $(g_\lambda^{ij})$ on $\mathcal M$.  By definition, 
	\begin{align}
		g_\lambda^{ij} = \frac 1{4\pi^2}\sum_{r,s=1}^\ell \pdf{z^i}{x^r}a^{rs}\pdf{z^j}{x^s}, \quad i, j = 1, \dots, \ell. \label{zh-07-11-1}
	\end{align}
	Here a factor $1/4\pi^2$ is introduced to simplify the expressions of $g_\lambda^{ij}$. 
	
We say that $\{z^1,\dots,z^\ell\}$ is a set of \emph{pencil generators} of $\mathscr A^W$ if the associated metric $g_\lambda=\bigl(g^{ij}_\lambda\bigr)$ depends	linearly on $\lambda$, i.e., $g^{ij}_\lambda$ can be represented in the form	
	\[g^{ij}_\lambda = g^{ij} + \lambda \eta^{ij},\] 
where $\eta=(\eta^{ij})$ is non-degenerate at generic points of $\mathcal{M}$,  and the 
Christoﬀel symbols $\Gamma^{ij}_{g_\lambda,k}=-g_{\lambda}^{ir}\Gamma_{g_{\lambda},rk}^j$ of the Levi-Civita connection of the contravariant metric $g_\lambda$ can be represented in terms of that of $g$ and $\eta$ as follows:
\begin{equation}\label{zzh-7}
\Gamma^{ij}_{g_{\lambda},k}=\Gamma^{ij}_{g,k}+\lambda \Gamma^{ij}_{\eta,k}.
\end{equation}
We know from \cite{JTZ2025-1} that for any set of pencil generators, the 
functions $g_\lambda^{ij}$ and $\Gamma^{ij}_{g_\lambda,k}$ are quasi-homogeneous polynomials of $z^1,\dots, z^\ell$ and $\lambda$ of degree
$\theta_i+\theta_j$ and $\theta_i+\theta_j-\theta_k$ respectively.

For a given set $\{z^1,\dots,z^\ell\}$ of pencil generators of $\mathscr A^W$, we consider the vector field
\begin{align}
		\label{Euler}
		E = \sum_{r=1}^{\ell}\theta_r z^r \pdf{}{z^r}.
	\end{align}
Denote
\begin{equation}\label{zh-05-25a}
	D_\eta = \{z \in \mathcal M~\big|~\det\eta(z) = 0\}.
	\end{equation}
Then $\eta=(\eta^{ij})$ induces a flat metric $(\eta_{ij})=(\eta^{ij})^{-1}$
on $\mathcal M\setminus D_\eta$, which we also denote by $\eta$.
We assume that we can choose a system of flat coordinates $t^1,\dots, t^\ell$ of the flat metric $\eta$, such that in these coordinates the vector field $E$ has the form
\begin{align}
		E = \sum_{\al, r=1}^\ell \theta_r z^r \pdf{t^\alpha}{z^r}\tangentvector{t^\alpha} = \sum_{\alpha=1}^\ell d_\alpha t^\alpha\tangentvector{t^\alpha},
	\end{align}
where $d_1,\dots,d_\ell$ are some real numbers. In such a case we call the vector field $E$ is diagonalizable. It is shown in  \cite{JTZ2025-1} that $d_\al$ must be positive numbers. 

Let us denote
\begin{align}\label{zh-07-20-1}
	D_0 = \{z \in \mathcal M \, |\, \exists\, i\in S,\, z^i = 0\},\quad D = D_\eta \cup D_0,
\end{align}
and $\mathcal M_D = \mathcal M \setminus D$, where 
	\begin{equation}\label{zzh-4}
		S=\{r\in\{1,\dots,\ell\}\mid m_r=(\omega, \alpha_r^\vee) > 0\},
	\end{equation}
then $\mathcal M_D$ is a dense open subset of $\mathcal M$. Let  $\Gamma^{\alpha\beta}_\gamma$ be the contravariant components of the Levi-Civita connection of $g$ in the flat coordinates $t^1,\dots, t^\ell$ of $\eta$. It is shown in 
\cite{JTZ2025-1} that one can define a Frobenius algebra structure 
on $T\left(\mathcal M_D\right)$ with the bilinear form $\langle\cdot\,, \cdot\rangle$ defined by the flat metric $\eta$, the multiplication
\begin{equation}\label{zh-07-17-1}
	\frac{\p}{\p t^\al}\cdot\frac{\p}{\p t^\beta}=c_{\al\beta}^\gamma\frac{\p}{\p t^\gamma},
\end{equation}
and the unit vector field
\begin{align}
	\label{unit}
	e=\eta^\sharp(\omega_e) = -\frac 1\kappa\mathrm{grad}_\eta \sum_{r=1}^{\ell}m_r \log z^r.
\end{align}
Here $m_r$ are defined in \eqref{omega = sum mr omegar},
\begin{equation}\label{zh-06-02e}
	c_{\al\beta}^\gamma=\frac{\kappa}{d_\rho}\eta_{\al\nu}
	\eta_{\beta\rho}\eta^{\gamma\zeta}\Gamma^{\nu\rho}_\zeta,\quad \omega_e = -\sum_{r=1}^{\ell}\frac1{\kappa}m_r \dif\log z^r,
\end{equation}
and $\eta^\sharp$ is the isomorphism
\begin{align}
	\eta^\sharp\colon T^*\mathcal M_D \to T^{**}\mathcal M_D = T\mathcal M_D,\quad \omega \mapsto \eta(\omega,\cdot).
\end{align}
We also assume here and in what follows summations over repeated upper and lower Greek indices.

\begin{thm}[\cite{JTZ2025-1}]
		\label{main}
		Suppose $\{z^1,\dots,z^\ell\}$ be a set of pencil generators associated with an irreducible reduced root system $R$ and a fixed weight $\omega$, and the vector field $E$ given by \eqref{Euler} is diagonalizable, then there exists a generalized Frobenius manifold structure of charge $d=1$ on $\mathcal M_D$, of which the flat metric is given by $\eta$ and the multiplication is defined by \eqref{zh-07-17-1}; moreover, 
the unit vector field $e$ is defined by\eqref{unit}, the Euler vector field $\tilde E=\frac1{\kappa} E$, and the intersection form coincides with $g$.
	\end{thm}

In this paper, we are to prove the following theorem.
\begin{thm}[Main Theorem]
For each $(R, \omega)=(A_\ell, \omega_\ell), ~(B_\ell, \omega_1),~ (C_\ell, \omega_1),~ (D_\ell, \omega_1)$, one can construct a set of pencil generators of $\mathscr A^W$, and a generalized Frobenius manifold structure on $\mathcal M_D(R, \omega)$ by using the approach proposed in Theorem \ref{main}. 
\end{thm}

We organize the paper as follows. In Sect.\,\ref{A-type} and Sect.\,\ref{C-type}, we prove the Main Theorem for the cases $(R, \omega)=(A_\ell, \omega_\ell)$ and $(R, \omega)=(C_\ell, \omega_1)$ respectively. In Sect.\,\ref{B, D-type} we show that that for the cases $(R, \omega)=(B_\ell, \omega_1)$ and $(D_\ell, \omega_1)$, the generalized Frobenius manifold structures that are constructed by using the approach of Theorem \ref{main}
are isomorphic to the ones obtained for the cases $(R, \omega)=(C_\ell, \omega_1)$.

\section{The Case of $(A_\ell, \omega_\ell)$ }\label{A-type}
	
	\subsection{The invariant $\lambda$-Fourier polynomial ring}\label{zh-07-12-1}
	Let $e_1,\dots,e_{\ell+1}$ be an orthonormal basis of $\mathbb{R}^{\ell+1}$, and $R$ be the root system of type $A_{\ell}$ in the hyperplane $V$ of $\mathbb{R}^{\ell+1}$ spanned the simple roots
\[\al_1=e_1-e_2,\dots,\al_\ell=e_\ell-e_{\ell+1}.\]
The coroots and the fundamental weights of $R$ are given by $\alpha_i^\vee =\alpha_i$ and
\begin{align*}
\omega_i =&\,\frac1{\ell+1}\left((\ell-i+1)\al_1+2(\ell-i+1)\al_2+\dots+(i-1)(\ell-i+1)\al_{i-1}\right.\\
&\left.+i(\ell-i+1)\al_i+i(\ell-i)\al_{i+1}+\dots+i\al_\ell\right),\quad
i=1,\dots,\ell.
\end{align*}
	
	Take $\omega=\omega_\ell$, then we have
%
\[\theta_j = (\omega_j,  \omega_\ell)=\frac{j}{\ell+1},\quad j=1,\dots,\ell,
    \]
and $\kappa=1$. We define $\xi^1,\dots, \xi^{\ell+1}$ by the relation
\[ c\omega + x^1\alpha_1^\vee + \dots + x^\ell\alpha_\ell^\vee=\xi^1 e_1+\dots+\xi^{\ell+1} e_{\ell+1},\]
then the basic generators of the $W_a(R)$-invariant $\lambda$-Fourier polynomial ring $\mathscr A^W$ can be represented in the form
\begin{align}\label{y j for A-type}
    y^j(\textbf{x})&=\lambda^{\frac{j}{\ell+1}}Y_j(\textbf{x})
    =\lambda^{\frac{j}{\ell+1}}\sigma_j(e^{2\pi i\xi^1}, \dots, e^{2\pi i\xi^{\ell+1}}), \quad j=1, \dots, \ell,
\end{align}
with 
\[\deg y^j=\theta_j=\frac{j}{\ell+1},\quad \deg\lambda=1,\quad j=1,\dots,\ell.\]
Here and in what follows we denote by $\sigma_j(u_1, \dots, u_{\ell+1})$ the $j$-th elementary symmetric polynomial of $u_1, \dots, u_{\ell+1}$ defined by
\begin{align}
    \prod_{i=1}^{\ell+1}(z+u_i)=\sum_{j=0}^{\ell+1} \sigma_j(u_1, \dots, u_{\ell+1})z^{\ell+1-j}.
\end{align}
We will also denote by $\sigma_j(u_1,\dots, \widehat{u}_k,\dots,u_{\ell+1})$, or simply by  $\sigma_j(\hat{u}_k)$, the 
$j$-th elementary symmetric polynomial of the $\ell$ variables $u_1, \dots, u_{k-1}, u_{k+1},\dots, u_{\ell+1}$, and by $\sigma_j(\widehat{u}_k, \widehat{u}_m)$ the 
$j$-th elementary symmetric polynomial of the $\ell-1$ variables $u_1, \dots, u_{k-1}, u_{k+1},\dots, u_{m-1}, u_{m+1},\dots, u_{\ell+1}$. 

\subsection{The pencil generators}
We are to show in this subsection that $\{y^1, \dots, y^\ell\}$ is a set of pencil generators of $\mathscr A^{W}$. 
To this end, let us consider the components $g_\lambda^{ij}$ of the metric $g_\lambda$ defined by \eqref{zh-07-11-1} with $z^i=y^i$. They can be represented in the form
\begin{align}\label{g expressed in xi}
        g_\lambda^{ij}=&\frac{1}{4\pi^2}\sum_{r,s=1}^\ell \frac{\p y^i}{\p x^r}a^{rs}\frac{\p y^j}{\p x^s}
        =\frac{1}{4\pi^2}\sum_{r,s=1}^{\ell+1} \frac{\p y^i}{\p \xi^r}b^{rs}\frac{\p y^j}{\p \xi^s},
    \end{align}
 where the matrices $(a^{rs})$ and $(b^{rs})$ are defined by
 \[(a^{ij})=\left((\alpha_i^\vee, \alpha_j^\vee)\right)^{-1}=\left(\frac{i(\ell+1-j)}{\ell+1}\right),\quad
        (b^{rs})=\frac 1{\ell+1}\begin{pmatrix}
			\ell & -1 & \cdots & -1\\
			-1 & \ell & \cdots & -1\\
			\vdots & \vdots & \ddots & \vdots\\
			-1 & -1 & \cdots & \ell
		\end{pmatrix}. 
    \]
\begin{lem}\label{g_lambda depends linearly on lambda}
    The quasi-homogeneous polynomials $g^{ij}_\lambda$ of $y^1,\dots, y^\ell, \lambda$ depend at most linearly on $\lambda$ when $\ell+1\le i+j\le2\ell$, and do not depend on $\lambda$ when $2\le i+j\le\ell$. 
\end{lem}
\begin{proof}
    Since 
$\deg g_\lambda^{ij}=\theta_i+\theta_j=\frac{i+j}{\ell+1}$, we have $1\le\deg g_\lambda^{ij}<2$ when $\ell+1\le i+j\le 2\ell$, and $\deg g_\lambda^{ij}<1$
when $2\le i+j\le \ell$. Thus from the fact that
$\deg y^j=\theta_j>0, \deg\lambda=1$, it follows that
$g^{ij}_\lambda$ depends at most linearly on $\lambda$ when $\ell+1\le i+j\le 2\ell$, and they do not depend on $\lambda$ when $2\le i+j\le \ell$. The lemma is proved.
\end{proof}
By using this lemma, we can represent $g^{ij}_\lambda$ in the form
\begin{align}
    g_\lambda^{ij}=g^{ij}+\lambda \eta^{ij},
\end{align}
where $\eta^{ij}=0$ when $2\le i+j\le \ell$, and the anti-diagonal elements of $\eta^{i, \ell+1-i}$ of $\eta$ are constants. 

\begin{prop}\label{counter-diagonal of eta}
    All the anti-diagonal elements of $\eta$ are equal to $\ell+1$. 
\end{prop}
\begin{proof}
Since the anti-diagonal elements of $\eta$ are constant, we can determine them by calculating $g^{ij}_\lambda$ for special values of $y^1,\dots, y^\ell$.

Let $e^{2\pi i\xi_0^1}, \dots, e^{2\pi i\xi_0^{\ell+1}}$ be the $\ell+1$ roots of $z^{\ell+1} + (-1)^{\ell+1}$. Then we have
\begin{align*}
\sigma_j(e^{2\pi i\xi_0^1}, \dots, e^{2\pi i\xi_0^{\ell+1}})=0,\quad  \sigma_{\ell+1}(e^{2\pi i\xi_0^1}, \dots, e^{2\pi i\xi_0^{\ell+1}})=1,\quad j=1,\dots,\ell.
\end{align*}
By using the identity
\begin{align*}
			\frac{z^{\ell+1} + (-1)^{\ell+1}}{z - e^{2\pi i\xi_0^k}} = z^\ell + e^{2\pi i\xi_0^k}z^{\ell-1} + \dots + e^{2\pi i\ell\xi_0^k}
			\end{align*}
we also have
\[\sigma_j(e^{2\pi i\xi_0^1},\dots,\widehat{e^{2\pi i\xi_0^k}},\dots,e^{2\pi i\xi_0^{\ell+1}}) = (-1)^je^{2\pi ij\xi_0^k},\quad j=1,\dots,\ell+1,\]
thus we obtain
\begin{align*}
				\left.\pdf{y^r}{\xi^k}\right|_{\xi=\xi_0} &= \left.2\pi i\mu^r e^{2\pi i\xi^k}\sigma_{r-1}(e^{2\pi i\xi^1},\dots,\widehat{e^{2\pi i\xi^k}},\dots,e^{2\pi i\xi^{\ell+1}})\right|_{\xi=\xi_0}
				= 2\pi i(-1)^{r-1}\mu^re^{2\pi ir\xi_0^k},
			\end{align*}
where $\mu=\lambda^{\frac{1}{\ell+1}}$. 
Now from \eqref{g expressed in xi} it follows that
           \begin{align*}
			g_\lambda^{rs} &=\left. \frac 1{4\pi^2(\ell+1)}\left((\ell+1)\sum_{k=1}^{\ell+1}\pdf{y^r}{\xi^k}\pdf{y^s}{\xi^k} - \sum_{j,  k=1}^{\ell+1}\pdf{y^r}{\xi^j}\pdf{y^s}{\xi^k}\right)\right|_{\xi=\xi_0}\\
			&= -(-1)^{\ell+1}\mu^{\ell+1}\left(\sum_{k=1}^{\ell+1}e^{2\pi i(r+s)\xi_0^k}
			- \frac 1{\ell+1}\sum_{j=1}^{\ell+1}e^{2\pi ir\xi_0^j}\sum_{k=1}^{\ell+1}e^{2\pi is\xi_0^k}\right)\\
			&= (\ell+1)\lambda
		\end{align*}
when $r+s=\ell+1$. The proposition is proved.
\end{proof}

From Proposition \ref{counter-diagonal of eta} we know that
\begin{align}\label{det of eta}
    \det(\eta^{ij})=(-1)^{\frac{\ell(\ell-1)}{2}}(\ell+1)^\ell, 
\end{align}
so $\eta$ is non-degenerate on $\mathcal{M}({A_\ell},\omega_\ell)$. 

\begin{thm}\label{zh-07-17-2}
The basic generators $y^1,\dots, y^\ell$ form a set of pencil generators of $\mathscr A^W$. 
\end{thm}

\begin{proof}
We only need to show that the 
Christoﬀel symbols $\Gamma^{ij}_{g_\lambda,k}$ of the Levi-Civita connection of the contravariant metric $g_\lambda$ depend linearly on $\lambda$. Indeed, this follows from the fact that they are quasi-homogeneous polynomials of $y^1,\dots, y^\ell$ and $\lambda$, and
\[\deg \Gamma_{\lambda, k}^{ij}=\theta_i+\theta_j-\theta_k=\frac{i+j-k}{\ell+1}<2.\]
The theorem is proved.
\end{proof}

\subsection{Flat coordinates of the metric $\eta$}\label{2.3}
We are to show in this subsection that one can choose a system of flat coordinates $t^1, \dots, t^\ell$ of the metric $\eta$ which are quasi-homogeneous polynomials in $y^1, \dots, y^\ell$. To this end, we first give the explicit expression of the components of the metric $\eta$.

\begin{prop}\label{zh-07-14-1}
In the coordinates $y^1,\dots, y^\ell$, the metric $\eta=(\eta^{ij})$ is given by
    \begin{align}\label{eta ab}
 \eta^{ij}=\begin{cases}0,& 2\le i+j\le \ell,\\
    (2\ell+2-i-j)y^{i+j-1-\ell}, &\ell+1\le i+j\le 2\ell,
    \end{cases}
\end{align}
and the contravariant components of the Levi-Civita connection of $\eta$
have the expressions
\begin{align}\label{Gamma lambda}
			\Gamma^{ij}_{\eta, k} = (\ell+1-j) \delta_{i+j-k,  \ell+1}. 
		\end{align}
\end{prop}
 Let us make some preparations for the proof of this proposition.

\begin{lem}\label{some basic properties}
   The symmetric polynomials $\sigma_a(u_1,\dots,u_{\ell+1})$ have the following properties: 
\begin{align*}
&1.\ \frac{\p}{\p u_i}\left(u_r\sigma_a(\widehat u_r)\right) = \begin{cases}
		u_r\sigma_{a-1}(\widehat u_r,  \widehat u_i), & r \not= i, \\
			\sigma_a(\widehat  u_r), & r=i,
		\end{cases} \ \mbox{for}\ a=1, \dots, \ell. \\
&2.\ \sum_{r=1}^{\ell+1} \sigma_a(\widehat{u}_r) = (\ell+1-a)\sigma_a(u_1,\dots,u_{\ell+1}),\quad a=1, \dots, \ell. \\
&3.\ \sum_{r=1, r\neq i}^{\ell+1} \sigma_a(\widehat{u}_r,  \widehat{u}_i) = (\ell-a)\sigma_a({\widehat{u}_i}),\quad a=1, \dots, \ell-1.\\
&4.\ \sum_{r=1}^{\ell+1} u_r\sigma_a(\widehat u_r) = (a+1)\sigma_{a+1}(u_1,\dots,u_{\ell+1}),\quad  a=1, \dots, \ell. \\
&5.\ \sum_{r=1, r\neq i}^{n} u_r\sigma_a(\widehat u_r, \widehat u_i) = (a+1)\sigma_{a+1}(\widehat u_i),\quad a=1, \dots, \ell-1.
\end{align*} 
Moreover, suppose $A$ is a subset of $\{u_1, \dots, u_{\ell+1}\}$ and $u_r \in A$, then we have
		\[u_r\sigma_a(\widehat A) = \sigma_{a+1}(\widehat{A\setminus u_r}) - \sigma_{a+1}(\widehat{A}).\] 
\end{lem}

Now take $u_j=\mu e^{2\pi i\xi^j}$ for $j=1, \dots, \ell+1$, where $\mu=\lambda^{\frac1{\ell+1}}$, then we have
\begin{align*}
    y^j&=\sigma_j(u_1, \dots, u_{\ell+1}),\quad  j=1, \dots, \ell;\\
    \lambda&=\sigma_{\ell+1}(u_1, \dots, u_{\ell+1}). 
\end{align*}
Thus $u_1,\dots,u_{\ell+1}$ are roots of the polynomial
\begin{align*}
f(u)=u^{\ell+1}-y^1 u^\ell+\dots+(-1)^\ell y^\ell u+(-1)^{\ell+1}\lambda,
\end{align*}
and as functions of $y^1,\dots, y^\ell, \lambda$, they satisfy the relations
\begin{align*}
    \frac{\partial u_j}{\partial\lambda}=\frac{(-1)^\ell}{f'(u_j)},\quad j=1,\dots,\ell+1.
\end{align*}
From \eqref{g expressed in xi} we know that
\begin{align}\label{g lambda ab}
    g_\lambda^{ab}=-\sum_{r, s=1}^{\ell+1}\frac{(\ell+1)\delta_{rs}-1}{\ell+1}u_r u_s\sigma_{a-1}(\widehat u_r)\sigma_{b-1}(\widehat u_s),
\end{align}
where $\delta_{rs}$ is the Kronecker-Delta function. Thus by using Lemma \ref{some basic properties} we obtain
\begin{align}
    &\eta^{a+1, b+1}=\frac{\partial}{\partial\lambda}g_\lambda^{a+1, b+1}\notag\\
    =&\,\frac{(-1)^{\ell+1}}{\ell+1}\sum_{j=1}^{\ell+1}\frac{1}{f'(u_j)}\frac{\partial}{\partial u_j}\left(\sum_{r, s=1}^{\ell+1}((\ell+1)\delta_{rs}-1)u_r u_s\sigma_{a}(\widehat u_r)\sigma_{b}(\widehat u_s)\right)\notag\\
=&\,\sum_{j=1}^{\ell+1}\frac{(-1)^{\ell+1}(F_j+G_j)}{(\ell+1)f'(u_j)},\label{eta a+1, b+1}
\end{align}
where $1\le a, b\le \ell-1$, and
\begin{align*}
F_j=&\,-(\ell-a)(\ell-b)\left[\sigma_{a+1}(u_1,\dots,u_{\ell+1})\sigma_b(\widehat u_j) \right.\\
&\left.+ \sigma_a(\widehat u_j)\sigma_{b+1}(u_1,\dots,u_{\ell+1}) \right],\\
G_j=&\, (\ell+1)\sum_{r \neq j}\left[\sigma_a(\widehat u_j, \widehat u_r)\sigma_{b+1}(\widehat u_r) + \sigma_{a+1}(\widehat u_r)\sigma_b(\widehat u_j, \widehat u_r)\right].
\end{align*}
In order to simplify the expression of \eqref{eta a+1, b+1} we need the following lemma.

\begin{lem}\label{1 term}
 For $1 \le a, b \le \ell$,  $a+b \ge \ell$,  we have
    \begin{align}
        \sum_{j=1}^{\ell+1}\frac{\sigma_a(\widehat u_j)}{f'(u_j)} =(-1)^{\ell}\delta_{a\ell}, \quad \sum_{j=1}^{\ell+1} \frac{\sigma_a(\widehat u_j)\sigma_b(\widehat u_j)}{f'(u_j)} =(-1)^{\ell}\sigma_{a+b-\ell}(u_1,\dots,u_{\ell+1}). 
    \end{align}
    	\end{lem}
\begin{proof}
 The first set of identities follows easily from the fact that
\[\sigma_a(\widehat u_j)=(-1)^a \res_{u=0}\frac{f(u)}{(u-u_j) u^{\ell+1-a}},\quad
\sum_{j=1}^{\ell+1}\frac{1}{f'(u_j) (u-u_j)}=\frac1{f(u)}.\]
The second set of identities also hold true, since
\begin{align*}
&\sum_{j=1}^{\ell+1} \frac{\sigma_a(\widehat u_j)\sigma_b(\widehat u_j)}{f'(u_j)}=\sum_{j=1}^{\ell+1}\frac1{f'(u_j)}\res_{z=0}\res_{u=0} (-1)^{a+b}\frac{f(z) f(u)}{(z-u_j)(u-u_j) z^{\ell+1-a} u^{\ell+1-b}}\\
=&\,\res_{z=0}\res_{u=0}(-1)^{a+b}\frac{f(z)}{z^{\ell+1-a}}\frac{f(u)}{u^{\ell+1-b}}\sum_{j=1}^{\ell+1}\frac1{f'(u_j)}\left[\frac1{z-u_j}-\frac1{u-u_j}\right]\frac1{u-z}\\
=&\,\res_{z=0}\res_{u=0}(-1)^{a+b}\frac{1}{z^{\ell+1-a}}\frac{1}{u^{\ell+1-b}}\frac1{u-z}\left[f(u)-f(z)\right]\\
=&\,\res_{z=0}\res_{u=0}(-1)^{a+b}\frac{1}{z^{\ell+1-a}}\frac{1}{u^{\ell+1-b}}\sum_{k=0}^\ell (-1)^k \sigma_k(u_1,\dots,u_{\ell+1})\sum_{p=0}^{\ell-k} u^{\ell-k-p} z^p\\
=&\,\sum_{k=0}^\ell\sum_{p=0}^{\ell-k}(-1)^{a+b+k}\sigma_k(u_1,\dots,u_{\ell+1})\delta_{\ell-a,p}\delta_{\ell-b,\ell-k-p}\\
=&\,(-1)^\ell\sigma_{a+b-\ell}(u_1,\dots,u_{\ell+1}).
\end{align*}
 The lemma is proved.
\end{proof}
\begin{proof}[Proof of Proposition \ref{zh-07-14-1}] It follows from Lemma \ref{1 term} that 
\begin{align*}
		\sum_{j=1}^{\ell+1} \frac{(-1)^{\ell+1}F_j}{f'(u_j)(\ell+1)} = 0, 
	\end{align*}
so we only need to compute
\begin{align*}
    \sum_{j=1}^{\ell+1} \frac{(-1)^{\ell+1}G_j}{f'(u_j)(\ell+1)}
		=\sum_{j=1}^{\ell+1}\frac{(-1)^{\ell+1}}{f'(u_j)}\sum_{r \not= j}\left(\sigma_a(\widehat u_j, \widehat u_r)\sigma_{b+1}(\widehat u_r) + \sigma_{a+1}(\widehat u_r)\sigma_b(\widehat u_j, \widehat u_r)\right)
\end{align*}
for $a, b$ satisfying the conditions $1\le a, b\le \ell-1$ and $\ell\le a+b\le 2\ell-2$.
By using the identities
\begin{align*}
    &\sigma_{a}(\widehat u_j, \widehat u_r)\sigma_{b+1}(\widehat u_r)=\left(\sigma_{a}(\widehat u_j) - w^r\sigma_{a-1}(\widehat u_j, \widehat u_r)\right)\sigma_{b+1}(\widehat u_r)\\
		=&\,\sigma_{a}(\widehat u_j)\sigma_{b+1}(\widehat u_r) - \sigma_{a-1}(\widehat u_j, \widehat u_r)\left(\sigma_{b+2}(u_1,\dots, u_{\ell+1}) - \sigma_{b+2}(\widehat u_r)\right)\\
		=&\,\sigma_{a}(\widehat u_j)\sigma_{b+1}(\widehat u_r) +\sigma_{a-1}(\widehat u_j, \widehat u_r) \sigma_{b+2}(\widehat u_r)  - \sigma_{a-1}(\widehat u_j, \widehat u_r)\sigma_{b+2}(u_1,\dots,u_{\ell+1}), 
\end{align*}
and by using Lemmas \ref{some basic properties}, \ref{1 term} we have
\begin{align*}
    &\sum_{j=1}^{\ell+1}\frac{(-1)^{\ell+1}}{f'(u_j)}\sum_{r \not= j}\sigma_a(\widehat u_j, \widehat u_r)\sigma_{b+1}(\widehat u_r) \\
    =&\,\sum_{j=1}^{\ell+1}\frac{(-1)^{\ell+1}}{f'(u_j)}\sum_{r \not= j}\sigma_{a-1}(\widehat u_j, \widehat u_r) \sigma_{b+2}(\widehat u_r)\\
    &+ \sum_{j=1}^{\ell+1}\frac{(-1)^{\ell+1}}{f'(u_j)}\sigma_{a}(\widehat u_j)\left((\ell-b)\sigma_{b+1}(u_1,\dots,u_{\ell+1}) - \sigma_{b+1}(\widehat u_j)\right)\\
		&-\sum_{j=1}^{\ell+1}\frac{(-1)^{\ell+1}}{f'(u_j)} (\ell-a+1)\sigma_{a-1}(\widehat u_j)\sigma_{b+2}(u_1,\dots,u_{\ell+1})\\
		=&\,\sum_{j=1}^{\ell+1}\frac{(-1)^{\ell+1}}{f'(u_j)}\sum_{r \not= j}\sigma_{a-1}(\widehat u_j, \widehat u_r) \sigma_{b+2}(\widehat u_r)
	+\sigma_{a+b+1-\ell}(u_1,\dots,u_{\ell+1})\\
=&\,\sum_{j=1}^{\ell+1}\frac{(-1)^{\ell+1}}{f'(u_j)}\sum_{r \neq j}\sigma_{a+b+1-\ell}(\widehat u_j, \widehat u_r) \sigma_{\ell}(\widehat u_r)
	+(\ell-b-1)\sigma_{a+b+1-\ell}(u_1,\dots,u_{\ell+1})\\
=&\,	(\ell-b)\sigma_{a+b+1-\ell}(u_1,\dots,u_{\ell+1}).
\end{align*}
Thus we arrive at
\[\eta^{a+1,b+1}=(2\ell-a-b)\sigma_{a+b+1-\ell}(u_1,\dots,u_{\ell+1})
=(2\ell-a-b) y^{a+b+1-\ell},\]
which yields the formula \eqref{eta ab}. 

Finally, The formulae \eqref{Gamma lambda} for the contravariant components of $\eta$ follow from the relations
\begin{align*}
    \frac{\p\eta^{ij}}{\p y^k} = \Gamma_{\eta,k}^{ij} + \Gamma_{\eta,k}^{ji}, \quad
				\eta^{is}\Gamma_{\eta,s}^{jk}=\eta^{js}\Gamma_{\eta,s}^{ik}. 
\end{align*}
The proposition is proved.
\end{proof}

\begin{thm}\label{flat metric of A}
    There exist quasi-homogeneous polynomials
    \begin{align*}
        t^\alpha=t^\alpha(y^1, \dots, y^\alpha), \quad \alpha=1, \dots, \ell
    \end{align*}
    of degrees $d_\alpha=\frac{\alpha}{\ell+1}$ such that $t^1, \dots, t^\ell$ are flat coordinates of $\eta$. Moreover,  the linear part of $t^\alpha$ is given by $y^\alpha$.  
    \begin{proof}
        The flat coordinates of $\eta$ are solutions to the system of equations
        \begin{align}\label{zh-07-16-2}
            \eta^{ik}\frac{\partial^2t}{\partial y^k\partial y^j}+\Gamma_{\eta, j}^{ik}\frac{\partial t}{\partial y^k}=0.
        \end{align}
Let $\psi_i=\frac{\p t}{\p y^i}$, then \eqref{zh-07-16-2} can be written as
the following system of equations for $\Psi=(\psi_1,\dots,\psi_\ell)$:\begin{equation}\label{zh-07-16-1}
\frac{\p\psi_i}{\p y^j}-\gamma^k_{ij}\psi_k=0,\quad i,j=1,\dots,\ell,
\end{equation}
where $\gamma^k_{ij}=-\eta_{ir}\Gamma_{\eta, j}^{rk}$.
From \eqref{eta ab} it follows that $\gamma^k_{ij}$ are quasi-homogeneous polynomials of $y^1,\dots, y^\ell$ of degree $\frac{k-i-j}{\ell+1}$, so we can find a fundamental system of solutions 
\[\Psi^\alpha=(\psi^\al_1,\dots, \psi^\al_\ell),\quad \al=1,\dots\ell,\]
of \eqref{zh-07-16-1} which satisfy the initial condition
\[\psi^i_j(0)=\delta^i_j,\quad i, j=1,\dots,\ell,\]
and are analytic at $(y^1,\dots, y^\ell)=(0,\dots,0)$.
Since the system of equations \eqref{zh-07-16-1} is invariant $\textit{w.r.t.}$ the transformation
        \begin{align*}
            y^j\to c^{\frac{j}{\ell+1}}y^j, \quad \psi_j\to c^{-\frac{j}{\ell+1}}\psi_j,\quad j=1, \dots, \ell
        \end{align*}
        for any nonzero constant $c$ and $\deg y^i=\frac{i}{\ell+1}>0$, the functions $\psi^\al_j$ must be quasi-homogeneous polynomials in $y^1, \dots, y^\ell$ of degrees $\frac{\al-j}{\ell+1}$. Thus we can choose the desired system of flat coordinates $t^1,\dots, t^\ell$ of the metric $\eta$
by using the relations $\frac{\p t^\al}{\p y^j}=\psi^\al_j$. The theorem is proved.
    \end{proof}
\end{thm}
\begin{cor}
    In the flat coordinates $t^1,\dots, t^\ell$, the components of the metric $\eta$ are given by
    \begin{align*}
        \eta^{\alpha\beta}=(\ell+1)\delta_{\beta, \ell+1-\alpha};
    \end{align*}
and the components of the metric $g$ and the Christoﬀel symbols of its Levi-Civita connection are quasi-homogeneous polynomials in $t^1, \dots, t^\ell$ with
\[\deg g^{\al\beta}(t)=\frac{\alpha+\beta}{\ell+1},\quad
\deg \Gamma^{\al\beta}_\gamma(t)=\frac{\alpha+\beta-\gamma}{\ell+1};\] 
moreover, the vector field define by \eqref{Euler} has the expression
\begin{equation}\label{zh-07-17-3}
E=\sum_{\al=1}^\ell d_\al t^\al \frac{\p}{\p t^\al}=\sum_{\al=1}^\ell \frac{\al}{\ell+1}t^\al \frac{\p}{\p t^\al}.
\end{equation}
\end{cor}
Note that if we define the involution 
\begin{align*}
    *\colon\{1, \dots, \ell\} \to \{1, \dots, \ell\}, \quad\alpha \mapsto \alpha^*:=\ell+1-\alpha, 
\end{align*}
then the degrees of $t^1, \dots, t^\ell$ satisfy the duality relation
\begin{align*}
    \deg t^\alpha+\deg t^{\alpha^*}=1. 
\end{align*}

\subsection{The generalized Frobenius manifold sturcture on $\mathcal{M}_D({A_\ell},\omega_\ell)$ and monodromy group}
From Theorems \ref{main}, \ref{zh-07-17-2} it follows that there is a generalized Frobenius manifold structure of charge $d=1$ on $\mathcal{M}_D({A_\ell},\omega_\ell)$ with flat metric $\eta$ and multiplication defined by \eqref{zh-07-17-1}. The unit vector field and the Euler vector field are given by
  \begin{align}\label{the unity for A-type1}
e=\eta^{\sharp}(\omega_e)=-\eta^\sharp(\mathrm{d}\log y^\ell)
    \end{align}
and by \eqref{zh-07-17-3}, and the intersection form is given by $g$. From Theorem \ref{flat metric of A} we also know that the contravariant components $\Gamma^{\al\beta}_\gamma(t)$ of the Levi-Civita connection of the intersection form $g$ are quasi-homogeneous polynomials of the flat coordinates $t^1,\dots, t^\ell$, so the structure constants $c^{\al\beta}_\gamma$ of the generalized Frobenius algebra defined by
\eqref{zh-06-02e} are also quasi-homogeneous polynomials of the flat coordinates.

The generalized Frobenius manifold structure can also be characterized by its potential $F(t)$, which is a quasi-homogeneous polynomial of degree 2 defined by
\begin{align*}
    \frac{\partial^3 F}{\partial t^\alpha\partial t^\beta\partial t^\gamma}=\eta_{\gamma\xi}c_{\alpha\beta}^\xi, \quad \al, \beta, \gamma=1,\dots,\ell.
\end{align*}
It can also be determined by using the intersection form $g$ as follows:
\[\frac{\p^2 F}{\p t^\al\p t^\beta}(t)=\frac1{2-d_\al-d_\beta}\,\eta_{\al\xi}\eta_{\beta\zeta}g^{\xi\zeta}(t)=\frac{\ell+1}{2\ell+2-\al-\beta}\, \eta_{\al\xi}\eta_{\beta\zeta}g^{\xi\zeta}(t),\quad \al,\beta=1,\dots,\ell.\]
From 
\cite{JTZ2025-1} we know that the monodromy group of $\mathcal{M}_D({A_\ell},\omega_\ell)$ is given by $\mathrm{Stab}_W(\omega_\ell)\ltimes\mathbb Z^\ell$, where
\begin{align}
	\mathrm{Stab}_W(\omega)=\{\sigma \in W(R)\mid\sigma(\omega) = \omega\}
\end{align}
 is a parabolic subgroup of $W(R)$ which can be represented by
 \begin{align}
	\mathrm{Stab}_W(\omega)=\langle\sigma_i\mid i=1,\dots,\ell-1\rangle\cong S_\ell.
\end{align}

\subsection{Examples}
Let us give some examples to illustrate the above construction of generalized Frobenius manifold structures on $\mathcal{M}_D(A_\ell, \omega_\ell)$. 

\begin{ex}
    Let $(R, \omega)=(A_1,\omega_1)$. We have the $W_a(R)$-invariant $\lambda$-Fourier polynomial 
		\[y^1=\lambda^{\frac{1}{2}}Y_1=e^{2\pi i x^1}+\lambda e^{-2\pi i x^1}. \]
The contravariant metric on $V\otimes\mathbb{C}$ is given by
		\[(\dif x^1, \dif x^1)=(\alpha^\vee_1, \alpha^\vee_1)^{-1}=\frac12,\]
which induces the metric 
		\[g^{11}_\lambda=-\frac12 (y^1)^2+2\lambda\]
on $\mathcal{M}$. In the flat coordinate $t^1=y^1$ of $\eta$, the flat pencil $\eta, g$ has the form
		\[\eta^{11}=2, \quad g^{11}=-\frac12 (t^1)^2. \]
		Thus we obtain a generalized Frobenius manifold structure with potential 
		\[F=-\frac1{96} (t^1)^4. \]
		The Euler vector field and the unity are given by
		\[E=\frac12 t^1\frac{\partial}{\partial t^1}, \quad
		e=-\frac2{t^1}\frac{\partial}{\partial t^1}. \]
\end{ex}
\begin{ex} Let $(R, \omega)=(A_2, \omega_2)$.
We have the $W_a(R)$-invariant $\lambda$-Fourier polynomials
		\begin{align*}
			y^1&=\lambda^{\frac13} Y_1=e^{2\pi i x^1}+e^{2\pi i (x^2-x^1)}+\lambda e^{-2\pi i x^2}, \\
			y^2&=\lambda^{\frac23} Y_2=e^{2\pi i x^2}+\lambda e^{-2\pi i x^1}+\lambda e^{2\pi i (x^1-x^2)}. 
		\end{align*}
The contravariant metric on $V\otimes\mathbb{C}$ is given by		
		\[\left((\alpha^\vee_i,  \alpha^\vee_j)\right)=\begin{pmatrix}2 &-1\\ -1 &2\end{pmatrix}, \quad
		\left(\dif x^i, \dif x^j\right)={\begin{pmatrix}2 &-1\\ -1 &2\end{pmatrix}}^{-1}=\frac13\begin{pmatrix}2 &1\\ 1 &2\end{pmatrix},\]
which induces the metric 
		\[\left(g_\lambda^{ij}\right)=\begin{pmatrix}-\frac23 (y^1)^2+2 y^2 & 3\lambda -\frac13 y^1 y^2\\[3pt] 3\lambda -\frac13 y^1 y^2& 2 \lambda y^1-\frac23 (y^2)^2\end{pmatrix}\]
on $\mathcal{M}$. We have the flat pencil 
		\[
		\left(\eta^{ij}\right)=\begin{pmatrix}0 &3\\ 3&2 y^1\end{pmatrix}, \quad
		\left(g^{ij}\right)=\begin{pmatrix}-\frac23 (y^1)^2+2 y^2 & -\frac13 y^1 y^2\\[3pt] -\frac13 y^1 y^2& -\frac23 (y^2)^2\end{pmatrix}. 
		\]
        The metric $\eta$ has flat coordinates
		\[t^1= y^1, \quad t^2=y^2-\frac16 (y^1)^2, \quad \]
		in which the metrics $\eta$ and $g$ have the form
		\[
		\left(\eta^{\alpha\beta}\right)=\begin{pmatrix}0 &3\\ 3&0\end{pmatrix}, \quad
		\left(g^{\alpha\beta}\right)=\begin{pmatrix} 2t^2-\frac13(t^1)^2& -t^1 t^2+\frac 1{18} (t^1)^3\\[4pt]-t^1 t^2+\frac 1{18} (t^1)^3&-\frac1{54}(6t^2-(t^1)^2)^2\end{pmatrix}. 
		\]
		Thus we obtain a generalized Frobenius manifold $\mathcal M_D(A_2, \omega_2)$ with potential
		\[
			F=\frac1{18} (t^2)^3-\frac1{36} (t^1)^2 (t^2)^2+\frac1{648}  (t^1)^4t^2-\frac1{19440} (t^1)^6.
		\]
The Euler vector field and the unity are given by
		\[E=\frac13 t^1\frac{\partial}{\partial t^1}+\frac23 t^2\frac{\partial}{\partial t^2}, \quad
		e=-\frac{18}{(t^1)^2+6 t^2}\frac{\partial}{\partial t^1}-\frac{6 t^1}{(t^1)^2+6 t^2 }\frac{\partial}{\partial t^2}. \]
	
\end{ex}
\begin{ex}
Let $(R, \omega)=(A_3, \omega_3)$.
We have the $W_a(R)$-invariant $\lambda$-Fourier polynomials
\begin{align*}
    y^1=&\,e^{2\pi i x^1}+e^{-2\pi i (x^1-x^2)}+\lambda e^{-2\pi i x^3}+e^{-2\pi i (x^2-x^3)}, \\
y^2=&\,\lambda^{\frac12} Y_2=e^{2\pi i x^2}+\lambda e^{-2\pi i x^2}+\lambda e^{2\pi i (x^1-x^3)}+\lambda e^{-2\pi i (x^1-x^2+x^3)}\\
&\,+e^{-2\pi i (x^1-x^3)}+e^{2\pi i (x^1-x^2+x^3)}, \\
y^3=&\,e^{2\pi i x^3}+\lambda e^{2\pi i (x^1-x^2)}+\lambda e^{-2\pi i x^1}+\lambda e^{2\pi i (x^2-x^3)}. 
\end{align*}
The contravariant metric on $V\otimes\mathbb{C}$ is given by	
\[ \left(\dif x_i, \dif x_j\right)={\left((\alpha^\vee_i,  \alpha^\vee_j)\right)}^{-1}
=\frac14\begin{pmatrix}3 &2&1\\ 2 &4&2\\1&2&3\end{pmatrix},\]
which induces the metric 
\begin{align*}
\left(g_\lambda^{ij}\right)
&=\begin{pmatrix}-\frac34 (y^1)^2+2 y^2 & -\frac12 y^1 y^2+3 y^3&-\frac14 y^1 y^3+4\lambda \\[3pt] -\frac12 y^1 y^2+3 y^3&-(y^2)^2+2 y^1 y^3+4\lambda&-\frac12 y^2y^3+3\lambda y^1\\[3pt]-\frac14 y^1 y^3+4\lambda&-\frac12 y^2y^3+3\lambda y^1&-\frac34 (y^3)^2+2 \lambda y^2
\end{pmatrix}. 
\end{align*}
on $\mathcal{M}$. We have the flat pencil 
\begin{align*}
\left(\eta^{ij}\right)=\begin{pmatrix}0&0 &4\\ 0&4&3 y^1\\4&3 y^1&2 y^2\end{pmatrix}, \quad \left(g^{ij}\right)=\begin{pmatrix}-\frac34 (y^1)^2+2 y^2 & -\frac12 y^1 y^2+3 y^3&-\frac14 y^1 y^3 \\[3pt] -\frac12 y^1 y^2+3 y^3&-(y^2)^2+2 y^1 y^3&-\frac12 y^2y^3\\[3pt]-\frac14 y^1 y^3&-\frac12 y^2y^3&-\frac34 (y^3)^2\end{pmatrix}. 
\end{align*}
The metric $\eta$ has flat coordinates
\[t^1= y^1, \quad t^2=y^2-\frac14 (y^1)^2, \quad 
t^3=y^3-\frac14 y^1 y^2+\frac5{96} (y^1)^3, \]
in which the metric $\eta$ has the form
\[
\left(\eta^{\alpha\beta}\right)=\begin{pmatrix}0&0 &4\\ 0&4&0\\ 4&0&0\end{pmatrix}, 
\]
and the intersection form is given by
\begin{align*}
    &g^{11}(t)=-\frac{1}{4}(t^1)^2+2t^2, ~g^{21}(t)=\frac{1}{32}((t^1)^3-24t^1t^2+96t^3), \\
    &g^{13}(t)=-\frac{1}{384}(t^1)^4+\frac{1}{8}(t^1)^2t^2-\frac{1}{2}(t^2)^2-t^1t^3, \\
    &g^{22}(t)=-\frac{1}{96}(t^1)^4+\frac{1}{4}(t^1)^2t^2-(t^2)^2-t^1t^3, \\
    &g^{23}(t)=\frac{5((t^1)^2-8t^2)((t^1)^3-24t^1t^2+96t^3)}{3072}, \\
&g^{33}(t)=\frac{-(t^1)^6+30(t^1)^4t^2-288(t^1)^2(t^2)^2-48(t^1)^3t^3+1152t^1t^2t^3+384(t^2)^3-2304(t^3)^2}{3072}. 
\end{align*}
We have the potential 
\begin{align*}F=&-\frac{(t^1)^8}{4128768}+\frac{(t^1)^6t^2}{73728}-\frac{(t^1)^5t^3}{30720}-\frac{(t^1)^4(t^2)^2}{3072}+\frac{(t^1)^3t^2t^3}{384}+\frac{(t^1)^2(t^2)^3}{384}\\
&-\frac{(t^1)^2(t^3)^2}{64}-\frac{t^1(t^2)^2t^3}{32}-\frac{(t^2)^4}{192}+\frac{t^2(t^3)^2}{8}. \end{align*}
of the generalized Frobenius manifold.
The Euler vector field is given by 
\[E=\frac14 t^1\frac{\partial}{\partial t^1}+\frac12 t^2\frac{\partial}{\partial t^2}+\frac34 t^3\frac{\partial}{\partial t^3}, \]
and the unity has the expression
\[e=-\frac{12}{(t^1)^3+24t^1t^2+96t^3}\left(32\tangentvector{t^1}+8t^1\tangentvector{t^2}+((t^1)^2+8 t^2)\tangentvector{t^3}\right). \]
\end{ex}

\subsection{The Landau-Ginzburg superpotential}
In this subsection, we will show that the above generalized Frobenius manifold structures on $\mathcal{M}_D({A_\ell},\omega_\ell)$ can be realized by using the following LG superpotentials:
\begin{equation}
    \Lambda(p)=p^{\ell+1}+a_1p^\ell+a_2p^{\ell-1}+\dots+a_\ell p,\quad a_1,\dots,a_\ell\in\mathbb{C},
\end{equation}
following the approach of constructing Frobenius manifold structures on Hurwitz spaces given in \cite{2d-tft}.

Let $\mathcal{M}_\ell$ be the space
\[\mathcal{M}_\ell=\left\{(a_1,\dots,a_\ell)\in\mathbb{C}^\ell\right\}.\]
We define on $\mathcal{M}_\ell$ the following $(0,2)$-type tensor
\begin{equation}\label{LG_eta}
    \tilde{\eta}(\p',\p'')=\sum_{q:\Lambda'(q)=0}\mathop{\mathrm{Res}}_{p=q}\frac{\p'(\Lambda)\p''(\Lambda)}{p^2\Lambda'(p)}\,\dd p,
\end{equation}
where the summation runs over the critical points of $\Lambda$. We also define a $(0,3)$-type tensor as follows:
\begin{equation}\label{LG_c}
    \tilde{c}(\p',\p'',\p''')=\sum_{q:\Lambda'(q)=0}\mathop{\mathrm{Res}}_{p=q}\frac{\p'(\Lambda)\p''(\Lambda)\p'''(\Lambda)}{p^2\Lambda'(p)}\,\dd p.
\end{equation}

\begin{lem}The tensor $\tilde{\eta}$ is a flat metric on $\mathcal{M}_\ell$, and it has flat coordinates
\begin{align}\label{flat coor. of tilde eta}
   \tilde{t}^\al=-\frac{\ell+1}{\alpha}\mathop{\mathrm{Res}}_{p=\infty}\left(\Lambda^{\frac{\alpha}{\ell+1}}(p)\frac{\mathrm{d}p}{p}\right)=a_\alpha+f_\alpha(a_1,\dots,a_{\alpha-1}),\quad\alpha=1,\dots,\ell.
    \end{align}
    Here $f_\alpha$ are homogeneous polynomials of $a_1,\dots,a_{\alpha-1}$ with 
    \[\deg a_\al=\frac{\al}{\ell+1},\quad \al=1,\dots,\ell.\]
\end{lem}
\begin{proof}
    Let $k$ be the following $(\ell+1)$-th root of $\Lambda(p)$:
    \[k=\Lambda^{\frac{1}{\ell+1}}=p+\frac{a_1}{\ell+1}+O\left(\frac 1p\right),\quad p\to\infty.\]
    By using the definition \eqref{flat coor. of tilde eta}, one can obtain
       \begin{equation}
        \log \frac{p}{k}=-\frac{1}{\ell+1}\sum_{\alpha=1}^\ell \frac{\tilde{t^\alpha}}{k^\alpha}+O\left(\frac{1}{k^{\ell+1}}\right),\quad k\to\infty,
    \end{equation}
    which implies that
    \[-\frac 1p\frac{\p p(k,\tilde{t})}{\p \tilde{t}^\al}=\frac{1}{\ell+1}\frac{1}{k^\al}+O\left(\frac{1}{k^{\ell+1}}\right),\quad k\to\infty.\]
    According to the implicit function theorem, we have
    \[
    \p_{\tilde{t}^\al}(\Lambda(p(k,\tilde{t}))=-\Lambda'(p)\frac{\p p(k,\tilde{t})}{\p \tilde{t}^\al},\quad\dd p=(\ell+1)k^\ell\dd k.
    \]
    So we have 
    \begin{equation}
        \tilde{\eta}(\p_{\tilde{t}^\al},\p_{\tilde{t}^\beta})=-\mathop{\mathrm{Res}}_{k=\infty}\frac{1}{p^2}\frac{\p p}{\p\tilde{t}^\al}\frac{\p p}{\p\tilde{t}^\beta}(\ell+1)k^\ell\dd k=\frac{1}{\ell+1}\delta_{\al+\beta,\ell+1}.
    \end{equation}The lemma is proved.
\end{proof}

Let us define an operation of multiplication on the tangent spaces of $\mathcal M_\ell$ by the following relation:
\begin{equation}\label{zh-01-07-1}\tilde{\eta}(\p'\cdot\p'',\p''')=\tilde{c}(\p',\p'',\p''').\end{equation}
We are to show that this operation yields a Frobenius manifold structure on $\mathcal{M}_\ell$. Suppose $q_1,\dots,q_\ell$ are the distinct critical points of $\Lambda(p)$, then we have
\[\Lambda'(p)=(\ell+1)\prod_{i=1}^\ell(p-q_i).\]
Let $u_i=\Lambda(q_i)$ be the critical values of $\Lambda(p)$. Note that $\p_{u_i}\Lambda(p)|_{p=q_j}=\delta_{ij}$, by using the Lagrange interpolation formula, we have
\begin{equation}
    \label{p alpha lambda}\p_{u_i}\Lambda(p)=\frac{p\Lambda'(p)}{(p-q_i)q_i\Lambda''(q_i)},
\end{equation}
which implies that
\begin{equation}\label{g(papb),eta(papb)}
    \tilde{\eta}(\p_{u_i},\p_{u_j})=\frac{\delta_{ij}}{q_{i}^2\Lambda''(q_i)},\quad\tilde{c}(\p_{u_i},\p_{u_j},\p_{u_k})=\frac{\delta_{ij}\delta_{ik}}{q_i^2\Lambda''(q_i)}.
\end{equation}
From \eqref{zh-01-07-1} it follows that \[\p_{u_i}\cdot\p_{u_j}=\delta_{ij}\p_{u_i},\]
thus $u_1,\dots,u_\ell$ are canonical coordinates of the multiplication. The unity vector field and the Euler vector field are given by
\[\tilde{e}=\sum_{i=1}^\ell\p_{u_i},\quad \tilde{E}=\sum_{i=1}^\ell u_i\p_{u_i}.\]
\begin{lem}
The unity vector field and the Euler vector field can also be represented in the form
\begin{equation}\label{e and E}
    \tilde{e}=-\mathrm{grad}_{\tilde{\eta}}\log a_\ell,\quad \tilde{E}=\sum_{\al=1}^\ell\frac{\al}{\ell+1}a_\al\p_{a_\al}=\sum_{\al=1}^\ell\frac{\al}{\ell+1}\tilde{t}^\al\p_{\tilde{t}^\al}.
\end{equation}
\end{lem}
\begin{proof}Note that
\[\p_{u_i}\Lambda(p)=\sum_{\al=1}^\ell(\p_{u_i}a_\al)p^{\ell+1-\al},\]
one can obtain from \eqref{p alpha lambda}  that
\[
\p_{u_i}a_\ell=\frac{(-1)^{\ell-1}(\ell+1)\sigma_\ell(q_i)}{q_i^2\Lambda''(q_i)}=\frac{-a_\ell}{q_i^2\Lambda''(q_i)}.
\]
By using \eqref{g(papb),eta(papb)}, we have
\[\tilde{e}=\sum_{i=1}^\ell\p_{u_i}=-\mathrm{grad}_{\tilde{\eta}}\log a_\ell.\]
The second formula comes from the homogeneity of $u_i$ and $\tilde{t}^\al$.
\end{proof}

\begin{lem}\label{25-12-17-1}
    If there locally exists a function $\varphi$, such that the unity can be represented as
    \begin{equation}\label{zh-01-08-1}
    \tilde{e}=\mathrm{grad}_\eta\varphi,\end{equation}
    then the $(0,4)$-type tensor $\tilde{c}_{\al\beta\gamma\xi}:=\partial_{\tilde{t}^\xi} \tilde{c}(\p_{\tilde{t}^\al},\p_{\tilde{t}^\beta},\p_{\tilde{t}^\gamma})$ is symmetric.
\end{lem}
\begin{proof}
    Denote $f_i=1/(q_i^2\Lambda''(q_i))$, then we have
    \[\tilde{\eta}(\p_{u_i},\p_{u_j})=f_i\delta_{ij},\quad\tilde{c}(\p_{u_i},\p_{u_j},\p_{u_k})=f_i\delta_{ij}\delta_{ik}.\]
    From the condition \eqref{zh-01-08-1} we see that
    \begin{equation}\label{zh-01-08-2}
    f_i=\frac{\p\varphi}{\p u_i},\quad i=1,\dots,\ell.
    \end{equation}
    We denote the Lam\'e coefficients and the rotation coefficients respectively by
    \[h_i=\sqrt{f_i},\quad\gamma_{ij}=\frac{1}{2\sqrt{f_if_j}}\frac{\p f_i}{\p u_j},\quad i\not=j.\]
    It follows from \eqref{zh-01-08-2} that $\gamma_{ij}=\gamma_{ji}$, i.e., $\tilde{\eta}$ is an Egoroff metric.

 Let $\mathrm{e}_i=\frac{1}{h_i}\p_{u_i}, i=1,\dots,\ell$ be the orthogonal frame. Suppose the flat frame has the following expansion w.r.t. the orthogonal frame:
    \begin{equation}\label{zh-01-08-3}  \p_{\tilde{t}^\al}=\sum_{i=1}^\ell\psi_{i\al}\mathrm{e}_i,\quad\al=1,\dots,\ell.\end{equation}
    By calculating the covariant derivatives $\nabla_{u_j}(\p_{\tilde{t}^\al})$, one can obtain
    \begin{align}
    &\frac{\p\psi_{j\al}}{\p u_i}=\gamma_{ji}\psi_{i\al},\quad i\not=j,\label{zh-01-08-4a}\\
    &\frac{\p\psi_{i\al}}{\p u_i}=-\sum_{k\not=i}\gamma_{ki}\psi_{k\al}.\label{zh-01-08-4b}
    \end{align}
 On the other hand, 
 \begin{equation}\label{zh-01-08-5}
 \tilde{c}_{\al\beta\gamma}=\tilde{c}(\p_{\tilde{t}^\al},\p_{\tilde{t}^\beta},\p_{\tilde{t}^\gamma})=\sum_{j=1}^\ell\frac{\psi_{j\al}\psi_{j\beta}\psi_{j\gamma}}{h_j}.\end{equation}
    By using \eqref{zh-01-08-3}-\eqref{zh-01-08-5}, we can show that
    \[\tilde{c}_{\al\beta\gamma\xi}-\tilde{c}_{\xi\beta\gamma\alpha}=\frac12\sum_{i\not=j}\frac{\gamma_{ij}-\gamma_{ji}}{h_ih_j}(\psi_{i\al}\psi_{j\xi}-\psi_{j\al}\psi_{i\xi})(\psi_{i\beta}\psi_{j\gamma}+\psi_{j\gamma}\psi_{i\beta}).\]
    Since $\gamma_{ij}=\gamma_{ji}$, we have $\tilde{c}_{\al\beta\gamma\xi}=\tilde{c}_{\xi\beta\gamma\alpha}$, the lemma is proved.
\end{proof}

From Lemma \ref{25-12-17-1}, we see that $(\mathcal{M}_\ell,\tilde{\eta},\cdot)$ gives a generalized Frobenius manifold structure. It is easy to verify that the Euler vector field $\tilde{E}$ satisfies the following relations:
\begin{align}
    &\mathcal{L}_{\tilde{E}}(\p'\cdot\p'')=\mathcal{L}_{\tilde{E}}\p'\cdot\p''+\p'\cdot\mathcal{L}_{\tilde{E}}\p''+\p'\cdot\p'',\label{condition-3-1}\\
    &\mathcal{L}_{\tilde{E}}\tilde{\eta}(\p',\p'')=\tilde{\eta}(\mathcal{L}_{\tilde{E}}\p',\p'')+\tilde{\eta}(\p',\mathcal{L}_{\tilde{E}}\p'')+\tilde{\eta}(\p',\p''),\label{condition-3-2}
\end{align}
so $\mathcal{M}_\ell$ is a conformal generalized Frobenius manifold with charge $d=1$. The intersection form of $\mathcal{M}_\ell$ can be represented by the following residue formula:
\begin{equation}\label{LG_g}
    \tilde{g}(\p',\p'')=\sum_{q:\Lambda'(q)=0}\mathop{\mathrm{Res}}_{p=q}\frac{\p'(\Lambda)\p''(\Lambda)}{p^2\Lambda(p)\Lambda'(p)}\,\dd p.
\end{equation}
It is defined on $\mathcal{M}_\ell\setminus\Sigma_0$, where 
\[\Sigma_0:=\{p | \Lambda(p)=0,\Lambda'(p)=0\}.\]
In terms of the canonical coordinates, we have
\begin{equation}
    \tilde{g}(\mathrm{d}u_i,\mathrm{d}u_j)=u_i q_i^2\Lambda''(q_i)\delta_{ij}.
\end{equation}

Let us proceed to show that the generalized Frobenius manifold structures on $\mathcal{M}_\ell$ and $\mathcal{M}_D(A_\ell,\omega_\ell)$ are isomorphic. We factorize the polynomial $\Lambda(p)$ as
\begin{equation}\label{2.6-21}
    \Lambda(p)=p\prod_{\al=1}^{\ell}(p-e^{i\phi_\al}),
\end{equation}
and let $h\colon\mathcal{M}_D(A_\ell,\omega_\ell)\to\mathcal{M}_\ell$ be induced by the map
\begin{equation}
(x^1,\dots,x^{\ell})\mapsto(\phi_1,\dots,\phi_{\ell}),
\end{equation}
where \begin{equation}\label{xi-to-phi}
    \phi_1=2\pi x^1,\quad\phi_\al=2\pi(x^\al-x^{\al-1}),\quad \al=2,\dots,\ell.
\end{equation}
From (\ref{zzh-2}),(\ref{y j for A-type}),(\ref{2.6-21}) and (\ref{xi-to-phi}), it follows that the map $h\colon(y^1,\dots,y^\ell)\mapsto(a_1,\dots,a_\ell)$ is given by
    \begin{equation}\label{a=y}
        a_\al=(-1)^\al y^\al |_{\lambda=0},\quad \al=1,\dots,\ell,
    \end{equation}
so we can know that $h$ is a diffeomorphism, and here the polynomial $\Lambda(p)$ coincides with $f_0(z):=f(z)|_{\lambda=0}$ in (\ref{2.3}).

\begin{thm}\label{9-1}
The map $h$ is an isomorphism between generalized Frobenius manifolds.
\end{thm}
\begin{proof}
From (\ref{e and E}) and (\ref{a=y}), it follows that 
\[h_*e=\tilde{e},\quad h_*{E}=\tilde{E}.\]
Since $\deg \tilde{t}^\al=d_\al>0$, we only need to prove that 
\[h^*\tilde{\eta}=\eta,\quad h^*\tilde{g}=g. \]

According to the definition of the canonical coordinates, we have 
\[\mathrm{d}u_i=\sum_{\al=1}^\ell q_i^{\ell+1-\al}\mathrm{d}a_\al.\]
 By using the inverse of the Vandermonde matrix, we obtain
\[\tilde{\eta}(\mathrm{d}a_\al,\dd a_\beta)=(\ell+1)^2\sum_{s=1}^\ell\frac{\sigma_{\al-1}(\widehat{q_s})\sigma_{\beta-1}(\widehat{q_s})}{\Lambda''(q_s)},\]
here the notations $\sigma_{\al-1}(\hat{q}_s), \sigma_{\beta-1}(\widehat{q_s})$ are introduced in Section \ref{zh-07-12-1}.

By using the same method that is employed in the proof of Lemma \ref{1 term}, we can show that
\begin{align*}
    \tilde{\eta}(\mathrm{d}a_\al,\dd a_\beta)=\begin{cases}
    (-1)^{\al+\beta-\ell-1}(2\ell+2-\al-\beta)a_{\al+\beta-\ell-1},&\al+\beta\ge\ell+2,\\
    \ell+1,&\al+\beta=\ell+1,\\
    0,&\al+\beta\le\ell,
    \end{cases}
\end{align*}
so we have $h^*\tilde{\eta}=\eta$.

From \eqref{p alpha lambda} and \eqref{2.6-21}, it follows that
\[\p_{u_i}\Lambda(p)=-\sum_{\al=1}^\ell \frac{ie^{i\phi_\al}\Lambda(p)}{p-e^{i\phi_\al}}\p_{u_i}\phi_\al,\quad\frac{\partial\phi_\al}{\partial u_i}=\frac{i}{(e^{i\phi_\al}-q_i)q_i\Lambda''(q_i)}.\]
Then we have
\begin{align*}
    \tilde{g}(\mathrm{d}\phi_\al,\mathrm{d}\phi_\beta)&=\sum_{j,k=1}^\ell\tilde{g}(\mathrm{d}u_j,\mathrm{d}u_k)\frac{\partial\phi_\al}{\partial u_j}\frac{\partial\phi_\beta}{\partial u_k}\\
    &=-\sum_{j=1}^\ell\frac{u_j}{(e^{i\phi_\al}-q_j)(e^{i\phi_\beta}-q_j)\Lambda''(q_j)}\\
    &=-\sum_{j=1}^\ell\mathop{\mathrm{Res}}_{p=q_j}\frac{\Lambda(p)}{(p-e^{i\phi_\al})(p-e^{i\phi_\beta})\Lambda'(p)}\mathrm{d}p\\
    &=\big(\mathop{\mathrm{Res}}_{p=e^{i\phi_\al}}+\mathop{\mathrm{Res}}_{p=e^{i\phi_\beta}}+\mathop{\mathrm{Res}}_{p=\infty}\big)\frac{\Lambda(p)}{(p-e^{i\phi_\al})(p-e^{i\phi_\beta})\Lambda'(p)}\mathrm{d}p\\
    &=\delta_{\alpha\beta}-\frac{1}{\ell+1},
\end{align*}
which implies that $h^*\tilde{g}=g$. The theorem is proved.
\end{proof}

\begin{rem}
The introduction of the generalized Frobenius manifold structure on $\mathcal{M}_\ell$ in terms of the residue formulae \eqref{LG_eta} and \eqref{LG_c} is also given by Zhonglun Cao in his Ph.D. thesis \cite{Z.Cao}. In his thesis he  attempted to construct this generalized Frobenius manifold structure by using the geometry of the orbit space of a certain extension of the affine Weyl group of type $A_\ell$, however, he did not provide a rigorous construction. 
He also showed in \cite{Z.Cao} that this generalized Frobenius manifold can be obtained from the dispersionless limit of the bihamiltonian structure of the $q$-deformed Gelfand–Dickey hierarchy \cite{FR}. 
\end{rem}

\section{The Case of $(C_\ell, \omega_1)$}\label{C-type}
\subsection{The $W_a(R)$-invariant $\lambda$-Fourier polynomial ring}\label{3.1}
Let  $R$ be the root system of type $C_\ell$ in an $\ell$-dimensional Euclidean space $V$ with orthonormal basis $e_1,\dots,e_\ell$. We take the simple roots
\[\al_1=e_1-e_2,\dots,\al_{\ell-1}=e_{\ell-1}-e_\ell,\quad \al_\ell=2 e_\ell.\]
The coroots and the fundamental weights are given by
\begin{align*}
&\al_j^\vee=\al_j,\quad \al_\ell^\vee=\frac12\al_\ell,\quad j=1,\dots,\ell-1.\\
&\omega_j=\al_1+2\al_2+\dots+(j-1)\al_{j-1}+j\left(\al_j+\dots+\al_{\ell-1}+\frac12\al_\ell\right),\quad j=1,\dots,\ell.
\end{align*}
Take $\omega=\omega_1$, then we have
\[\theta_j=(\omega_j,\omega_1)=1,\quad j=1,\dots,\ell,\]
and $\kappa=1$. We define $\xi^1,\dots, \xi^{\ell}$ by the relation
\[ c\omega_1 + x^1\alpha_1^\vee + \dots + x^\ell\alpha_\ell^\vee= c\omega_1 +\xi^1 e_1+\dots+\xi^\ell e_\ell,\]
and denote 
\begin{align*}
    \zeta^j=e^{2\pi i\xi^j}+e^{-2\pi i\xi^j}, \quad j=1, \dots, \ell.
\end{align*}
The basic generators of the $W_a(R)$-invariant $\lambda$-Fourier polynomial ring $\mathscr A^W$ can be represented in the form \cite{bourbaki}
\begin{align}
    y^j:=y^j(\textbf{x})=\lambda \sigma_j(\zeta^1,\dots,\zeta^\ell), \quad j=1, \dots, \ell,
\end{align}
here $\lambda=e^{-2\pi i c}$. Let $y^0=\lambda$, then we have the following generating function for $y^1,\dots, y^\ell$:
\begin{align}\label{P(u)}
P(u)=\sum_{j=0}^\ell y^ju^{\ell-j}=\lambda\prod_{k=1}^\ell(u+\zeta^k).
\end{align}

\subsection{The pencil generators}
Unlike the $(R,\omega)=(A_\ell, \omega_\ell)$ case that we studied in the last section, the basic $y^1, \dots, y^\ell$ are not pencil generators of $\mathscr A^{W}$. In order to find a set of pencil generators of $\mathscr A^{W}$, we need to compute explicitly the metric $g_\lambda$ defined by \eqref{zh-07-11-1} and the contravariant components of its Levi-Civita connection.

\begin{lem}
 The following formulae hold true for the generating functions of the metric $g_\lambda$ and the contravariant components of its Levi-Civita connection
 in the coordinates $y^1,\dots, y^\ell$:
    \begin{align}
        \sum_{i, j=1}^\ell g_\lambda^{ij}(y)u^{\ell-i}v^{\ell-j}
        =&-\ell P(u)P(v)+\frac{u^2-4}{u-v}P'(u)P(v)-\frac{v^2-4}{u-v}P(u)P'(v), \label{generating of g lambda ij}\\
        \sum_{i, j, k=1}^\ell\Gamma^{ij}_{\lambda, k}(y)\mathrm{d}y^ku^{\ell-i}v^{\ell-j}
        =&-\ell P(u)\mathrm{d}P(v)+\frac{u^2-4}{u-v}P'(u)\mathrm{d}P(v)-\frac{v^2-4}{u-v}P(u)\mathrm{d}P'(v)\notag\\
        &+\frac{uv-4}{(u-v)^2}\left(P(v)\mathrm{d}P(u)-P(u)\mathrm{d}P(v)\right),\label{generating of Gamma lambda k ij}
        \end{align}
where $P'(u)=\frac{\p P}{\p u}(u),\, P'(v)=\frac{\p P}{\p v}(v)$.
\begin{proof}(cf.  \cite{Advance})
From \eqref{zh-07-17-5} we know that the contravariant metric $a$ defined on $V\otimes \mathbb{C}$ has the property
\[a(\dif\xi^i,\dif\xi^j)=\delta_{ij},\quad i, j=1,\dots,\ell.\]
Thus by using the identities
\begin{align}
    \frac{\partial P}{\partial \xi^k}(u)&=\frac{2\pi iP(u)(e^{2\pi i\xi^k}-e^{-2\pi i\xi^k})}{u+\zeta^k}, \quad 1\le k \le \ell, \label{partial P partial xi a}\\
    P'(u)&=P(u)\sum_{k=1}^\ell\frac{1}{u+\zeta^k}\label{P'(u)},
\end{align}
we obtain
    \begin{align*}
        &\sum_{j,k=0}^\ell\big((\varphi_\lambda)_*a\big)(\mathrm{d}y^j, \mathrm{d}y^k)u^{\ell-j}v^{\ell-k}=\frac{1}{4\pi^2}\sum_{a=1}^\ell\frac{\partial P(u)}{\partial \xi^a}\frac{\partial P(v)}{\partial \xi^a}\\
        =&-\sum_{a=1}^\ell P(u)P(v)\frac{(\zeta^a)^2-4}{(u+\zeta^a)(v+\zeta^a)}\\
        =&-\sum_{a=1}^\ell P(u)P(v)(1-\frac{u^2-4}{u-v}\frac{1}{u+\zeta^a}+\frac{v^2-4}{u-v}\frac{1}{v+\zeta^a})\\
        =&-\ell P(u)P(v)+\frac{u^2-4}{u-v}P'(u)P(v)-\frac{v^2-4}{u-v}P(u)P'(v),
    \end{align*}
so we proved the first formula \eqref{generating of g lambda ij}. In a similar way we can prove the second formula \eqref{generating of Gamma lambda k ij}. The lemma is proved. 
\end{proof}
\end{lem}
  The above lemma shows that $g_\lambda^{ij}(y)$ are quadratic polynomials in $y^1,\dots, y^\ell$ and $\lambda$, which may not depend linearly on $\lambda$, so in general $y^1, \dots, y^\ell$ are not pencil generators of $\mathscr A^W$. Since 
\[\deg y^j=\deg\lambda=1, \quad j=1, \dots, \ell,\] 
we may attempt to construct proper generators of the form
  \begin{align}\label{z j}
      z^j=y^j+c_j\lambda, \quad j=1, \dots, \ell, 
  \end{align}
such that $z^1, \dots, z^\ell$ form a set of pencil generators of $\mathscr A^{W}$. 
The following theorem shows that we can indeed find pencil generators in this way.
\begin{thm}\label{pencil generator of C-type}
    For any fixed number $0\le m\le \ell$, there exist pencil generators $z^1, \dots, z^\ell$ of the form \eqref{z j}, with constants $c_1, \dots, c_\ell$ defined by the generating function
    \begin{align}
        P_0(u)=\sum_{j=1}^\ell c_ju^{\ell-j}=u^\ell-(u+2)^m(u-2)^{\ell-m}. 
    \end{align}
    \end{thm}
    \begin{proof}(cf. \cite{Advance})
It suffices to find a polynomial $P_0(u)=\sum_{j=1}^\ell c_ju^{\ell-j}$ such that, after the shift
    \[P(u)\to P(u)-\lambda P_0(u),\quad  P(v)\to P(v)-\lambda P_0(v), \]
    the right hand side of \eqref{generating of g lambda ij} and \eqref{generating of Gamma lambda k ij} depend at most linearly on $\lambda$. This condition is equivalent to the following equation for $P_0$:
    \begin{align}\label{equation of P0}
        -\ell P_1(u)P_1(v)+\frac{u^2-4}{u-v}P_1'(u)P_1(v)-\frac{v^2-4}{u-v}P_1(u)P_1'(v)=0,
    \end{align}
where $P_1(u)=u^\ell-P_0(u)$.
From the proof of Theorem 3.4 of \cite{Advance} we know that we can take
\begin{equation}\label{01-18-c}
P_1(u)=(u+2)^m(u-2)^{\ell-m}
\end{equation}
for any fixed $0\le m\le \ell$.

In order to prove that 
\begin{equation}\label{zh-07-20-2}
g=(g^{ij}(z))=\left(g_\lambda(\dd z^i,\dd z^j)\right)|_{\lambda=0},\quad \eta=(\eta^{ij}(z))=\left(\frac{\partial}{\partial\lambda}g_\lambda(\dd z^i, \dd z^j)\right)\end{equation}
form a flat pencil of metrics, we need to show that the determinant of $(\eta^{ij}(z))$ does not vanishes at generic point of the orbit space $\mathcal{M}(C_\ell,\omega_1)$ of the affine Weyl group. Define the following generating function of the new coordinates $z^1,\dots, z^\ell$ that are introduced in \eqref{z j}:
\begin{equation}\label{01-18-d}Q(u)=\sum_{j=1}^\ell z^j u^{\ell-j},\end{equation}
then in the coordinates $z^1,\dots,z^\ell$ the metric $\eta$ can be represented by
\begin{align}
&\sum_{j,k=1}^\ell\eta\left(\dd z^j,\dd z^k\right)u^{\ell-j}v^{\ell-k}
=-\ell \left(Q(u) P_1(v)+Q(v) P_1(u)\right)\notag\\
&\quad +\frac{u^2-4}{u-v}\left(Q'(u) P_1(v)+P_1'(u) Q(v)\right)-\frac{v^2-4}{u-v}\left(Q'(v) P_1(u)+P_1'(v) Q(u)\right).\label{zh-07-19-1}
\end{align}
To prove the non-degeneracy of the metric $\eta$, 
let us adopt the method of calculation for the metric $\eta$ given in the proof of Theorem 3.4 of \cite{Advance} to the present case. For any fixed $0\le m\le \ell$, consider the linear change of coordinates
\[(z^1, \dots, z^\ell)\mapsto(\tau^1, \dots, \tau^\ell)\]
defined by
\begin{align}\label{generating function of tau j}
    \sum_{j=1}^\ell z^ju^{\ell-j}=\sum_{j=1}^{\ell-m}\tau^j(u+2)^m(u-2)^{\ell-m-j}-\sum_{j=\ell-m+1}^\ell\tau^j(u+2)^{\ell-j}(u-2)^{j-1}. 
\end{align}
By inserting the expressions for $Q(u), Q(v)$ into both sides of \eqref{zh-07-19-1}, we obtain an identity in the variables $u, v$. Dividing this identity by $P_1(u)$ and $P_1(v)$ yields a new identity relating two rational functions in $u$ (treating $v$ as a parameter), which possess poles at $u= \pm 2$. Comparing the regular and singular parts of this expression at $u=2$ and $u=-2$ leads to explicit formulae for the matrix elements of $\left(\eta(d\tau^i, d\tau^j)\right)$. This matrix exhibits a block diagonal form
\begin{align}\label{eta(tau)}
    \begin{pmatrix}
        W_1 &0\\
        0&W_2
    \end{pmatrix},
\end{align}
with anti-triangular matrices
\begin{equation}\label{expansions of W1, W2 in tau}
    W_1=\begin{pmatrix}
        R_1&R_2&\cdots&R_{\ell-m}\\
        R_2&R_3&\cdots&0\\
        \vdots&\vdots& &\vdots\\
        R_{\ell-m-1} &R_{\ell-m}& \cdots&0\\
        R_{\ell-m}&0&\cdots&0
    \end{pmatrix}, \quad W_2=\begin{pmatrix}
        S_1&S_2&\cdots&S_m\\
        S_2&S_3&\cdots&0\\
        \vdots&\vdots& &\vdots\\
        S_{m-1}&S_m&\cdots&0\\
        S_m&0&\cdots&0
    \end{pmatrix},
\end{equation}
they have entries 
\begin{align}
    R_s=4s\tau^{s}+(1-\delta_{s, \ell-m})(s+1)\tau^{s+1}, \quad
    S_r=4r\tau^{\ell-m+r}-4(1-\delta_{r, m})r\tau^{\ell-m+r+1}\label{Qs Sr in tau}
    \end{align}
for $1\le s\le \ell-m,\, 1\le r\le m$. Thus
by a simple computation we get
\begin{equation}\label{det of eta in C-case}
\det(\eta^{ij}(\tau))=\begin{cases}
    (-1)^{\frac{\ell^2-(2m+1)\ell+2m^2}{2}}4^\ell m^m(\ell-m)^{\ell-m}(\tau^{\ell-m})^{\ell-m}(\tau^\ell)^m, &m\neq 0, \ell;\\
    (-1)^{\frac{\ell(\ell-1)}{2}}4^\ell\ell^\ell(\tau^\ell)^\ell, &m=0, \ell. 
\end{cases}
\end{equation}
So the matrix $(\eta^{ij}(\tau))$ is non-degenerate on 
\begin{equation}\label{01-18-a}
\mathcal{M}_{\ell,m}:=\mathcal{M}(C_\ell,\omega_1)\setminus(\{\tau^\ell=0\}\cup\{\tau^{\ell-m}=0\})
\end{equation}
when $m\not=0, \ell$, and on 
\[\mathcal{M}_{\ell,0}=\mathcal{M}_{\ell,\ell}:=\mathcal{M}(C_\ell,\omega_1)\setminus\{\tau_\ell=0\}\] 
when $m=0, \ell$. 
Thus $\left(g^{ij}(z)\right)$ and $\left(\eta^{ij}(z)\right)$ form a flat pencil of metrics on $\mathcal{M}_m$, and $z^1, \dots, z^\ell$ are pencil generators of $\mathscr A^W$. The theorem is proved. 
\end{proof}

\begin{cor}
    In the coordinates $\tau^1, \dots, \tau^\ell$, the components of the metric $g=\left(g^{ij}(\tau)\right)$ and the contravariant components $\Gamma_k^{ij}(\tau)$ of its Levi-Civita connection are quasi-homogeneous polynomials with degrees
    \begin{align}
        \deg g^{ij}(\tau)=i+j, \quad \deg\Gamma_k^{ij}(\tau)=i+j-k,
    \end{align}
and we have $\deg \tau^j=j$. 
\end{cor}

\subsection{Flat coordinates of the metric $\eta$}
In this subsection, we are to show that the flat coordinates of the metric $\eta=(\eta^{ij}(z))$ defined in the last subsection are algebraic functions of $\tau^1, \dots, \tau^\ell$. Since the explicit form of the matrix $\left(\eta(d\tau^i, d\tau^j)\right)$ given by \eqref{eta(tau)}--\eqref{Qs Sr in tau} coincides with the 
\[\begin{pmatrix} W_2 &0 \\ 0 & W_3\end{pmatrix}\] 
block of the matrix $\left(\eta^{ij}(\tau)\right)$ given in (3.26) of \cite{Advance},
we can adopt directly the results of Lemma 3.8, Lemma 3.9 and Theorem 3.11 of \cite{Advance}, by setting the parameter $k$ that appears there to be zero, to characterize properties of flat coordinates of $\eta$.

We first perform changes of coordinates to simplify the matrix $(\eta^{ij}(\tau))$. 
\begin{lem}[cf. Lemma 3.8 of \cite{Advance}]\label{tau->w}
    There exists a system of coordinates $w^1, 
    \dots, w^\ell$ of the form
   \begin{align*}
        &w^j=\tau^j+\sum_{s=j+1}^{\ell-m}c_s^j\tau^s, \quad 1\le j\le \ell-m-1, \\
        &w^j=\tau^j+\sum_{s=j+1}^\ell h_s^j\tau^s, \quad \ell-m+1\le j\le \ell-1, \\
        &w^{\ell-m}=\tau^{\ell-m}, \quad w^\ell=\tau^\ell. &
    \end{align*}
    with certain constants $c_s^j, h_s^j$, such that the the matrix $\left(\eta(\mathrm{d}w^i, \mathrm{d}w^j)\right)$ still possesses block diagonal form (\ref{eta(tau)})--(\ref{expansions of W1, W2 in tau}) with the entries replaced by
    \begin{align*}
        R_s=4sw^s, \quad
        S_r=4rw^{\ell-m+r},\quad 1\le s\le \ell-m,\ 1\le r\le m.
        \end{align*}
\end{lem}

 The following lemma simplifies the expression of the metric $\eta$ further.
\begin{lem}[cf. Lemma 3.9 of \cite{Advance}]\label{w->v}
    In the new coordinates $v^1,\dots, v^\ell$ defined by
    \begin{align*}
        &v^1=w^1(w^{\ell-m})^{-\frac{1}{2(\ell-m)}}, \quad v^\ell=(w^\ell)^{\frac{1}{2m}},\\
        &v^s=w^s(w^{\ell-m})^{-\frac{s}{\ell-m}}, \quad 2\le s\le \ell-m-1, \\
        &v^{\ell-m}=(w^{\ell-m})^{\frac{1}{2(\ell-m)}}, \quad v^{\ell-m+1}=w^{\ell-m+1}(w^\ell)^{-\frac{1}{2m}}, \\
        &v^{r}=w^r(w^\ell)^{-\frac{r+m-\ell}{m}},\quad  \ell-m+2\le r\le \ell-1. 
    \end{align*}
the metric $\eta$ has the expression
    \begin{equation}\label{eta (v)}
\begin{pmatrix}
    B_1&0\\
    0&B_2
\end{pmatrix},
    \end{equation}
where $B_1, B_2$ are anti-triangular matrices of the form
\begin{align*}
    B_1&=\begin{pmatrix}
        0&0&0&0&\cdots&0&2\\
        0&H_3&H_4&\cdots&H_{\ell-m-1}&H_{\ell-m}& \\
        0&H_{4}&H_5&\cdots&H_{\ell-m}& & \\
        \vdots&\vdots&\vdots& \begin{rotate}{45}
			$\cdots$
		\end{rotate} & & & \\
        0&H_{\ell-m-1}&H_{\ell-m}& & & & \\
        0&H_{\ell-m}& & & & & \\
        2& & & & & &  
    \end{pmatrix},\\[8pt]
    B_2&=\begin{pmatrix}
        0&0&0&0&\cdots&0&2\\
        0&H_{\ell-m+3}&H_{\ell-m+4}&\cdots&H_{\ell-1}&H_{\ell}& \\
        0&H_{\ell-m+4}&H_{\ell-m+5}&\cdots&H_{\ell}& & \\
        \vdots&\vdots&\vdots& \begin{rotate}{45}
			$\cdots$
		\end{rotate}& & & \\
         0&H_{\ell-1}&H_{\ell}& & & & \\
        0&H_\ell& & & & & \\
        2& & & & & &  
    \end{pmatrix},
\end{align*}
    with
    \begin{align*}
        &H_{s}=4s(v^{\ell-m})^{-2}v^s,\quad H_{\ell-m}=4(\ell-m)(v^{\ell-m})^{-2}, \notag\\
        &H_{\ell-m+j}=4j(v^\ell)^{-2}v^{\ell-m+j}, \quad H_\ell=4m(v^\ell)^{-2}, \\
        &3\le s\le \ell-m-1, \quad 3\le j\le m-1. 
    \end{align*}
\end{lem}
\begin{rem}\label{m=0, m=1, m=2, B 2 is constant}
    When $m=0$ (resp. $m=\ell$), the matrix $B_2$ (resp. $B_1$)does not appear in \eqref{eta (v)}. When $m=1$ (resp. $m=\ell-1$), we have $B_2=1$ (resp. $B_1=1$). When $m=2$ (resp.  $m=\ell-2$), the matrix $B_2$ (resp.  $B_1$) has the form $\begin{pmatrix}
        0&2\\2&0
    \end{pmatrix}$. 
\end{rem}
\begin{thm}[cf. Theorem 3.11 of \cite{Advance}]\label{flat coordinates of C}
    One can choose flat coordinates of the metric $\eta$ of the form
    \begin{align*}
        &t^1=v^1+v^{\ell-m}h_1(v^2, \dots, v^{\ell-m-1}), \\
        &t^\alpha=v^{\ell-m}(v^\alpha+h_\alpha(v^{\alpha+1}, \dots, v^{\ell-m-1})), \quad 2\le \alpha\le \ell-m-1, \\
        & t^{\ell-m}=v^{\ell-m}, \\
        &t^{\ell-m+1}=v^{\ell-m+1}+v^\ell h_{\ell-m+1}(v^{\ell-m+2}, \dots, v^{\ell-1}), \\
        &t^\beta=v^\ell(v^\beta+h_\beta(v^{\beta+1}, \dots, v^{\ell-1})), \quad \ell-m+2\le \beta\le \ell-1, \\
        &t^\ell=v^\ell. 
    \end{align*}
Here $h_{\ell-m-1}=h_{\ell-1}=0$, $h_\alpha$ are quasi-homogeneous polynomials of degree $\frac{\ell-m-\alpha}{\ell-m}$ for $1\le \alpha\le \ell-m-2$, and $h_\beta$ are quasi-homogeneous polynomials of degree $\frac{\ell-\beta}{m}$ for $\ell-m+1\le \beta\le \ell-2$. 
\end{thm}

From the  above-mentioned construction of the flat coordinates $t^1,\dots, t^\ell$, we know that they are quasi-homogeneous functions of $z^1,\dots, z^\ell$ with degrees
\begin{align}
    d_\alpha&=\deg t^\alpha=\frac{2(\ell-m-\alpha)+1}{2(\ell-m)}, \quad 1\le \alpha\le \ell-m;\label{zh-07-20-3}\\
    d_\beta&=\deg t^\beta=\frac{2(\ell-\beta)+1}{2m}, \quad \ell-m+1\le \beta\le \ell. \label{zh-07-20-4}
\end{align}
These numbers satisfy a duality relation which is similar to that of \cite{DZ1998} and \cite{Advance}. To describe this relation, let $\mathcal R$ be the Dynkin diagram of type $C_\ell$. For any given integer $0\le m\le \ell$, we separate $\mathcal R$ into two componnets, the first one is formed by the first $\ell-m$ vertices, and the second one is formed by the remaining $m$ vertices. On each component, we have an involution $\beta\mapsto \beta^*$ given by the reflection with respect to its center. Then we have
\begin{equation}
    d_\beta+d_{\beta^*}=1,\quad \beta=1, \dots, \ell, 
\end{equation}
and $\eta^{\alpha\beta}(t)$ is a nonzero constant if and only if $\beta=\alpha^*$. 

We have the following corollaries. 
\begin{cor}\label{zh-02-09-5}
    In the flat coordinates $t^1, \dots, t^{\ell+1}$, the matrix $\left(\eta^{\alpha\beta}(t)\right)$ has the form
    \begin{align}
        \begin{pmatrix}
            A_1&0\\0&A_2
        \end{pmatrix},
    \end{align}
where $A_1, A_2$ are $(\ell-m)\times(\ell-m)$ and $m\times m$ matrices respectively, and they have the form
    \begin{align}
        A_1=\begin{pmatrix}
            & & & &2& \\
            & & &4(\ell-m)& & \\
            & &\begin{sideways}$\ddots$\end{sideways}& & &\\
            &4(\ell-m)& & & &\\
            2& & & & &
        \end{pmatrix}, \quad A_2=\begin{pmatrix}
            & & & &2& \\
            & & &4m& & \\
            & &\begin{sideways}$\ddots$\end{sideways}& & &\\
            &4m& & & &\\
            2& & & & &
        \end{pmatrix}
    \end{align}
when $m, \ell-m\neq 0, 1, 2$; when $m=0$ (resp.  $m=\ell$), the matrix $A_2$ (resp. $A_1$) does not appear; when $m=1$ (resp. $m=\ell-1$), we have $A_2=1$ (resp. $A_1=1$); when $m=2$ or $m=\ell-2$, we have 
\[A_2=\begin{pmatrix}
        0&2\\2&0
    \end{pmatrix} \ \mbox{or}\  A_1=\begin{pmatrix}
        0&2\\2&0
    \end{pmatrix}. \]
\end{cor}
\begin{cor}\label{polynomial property of g and Gamma}
    In the flat coordinates $t^1, \dots, t^{\ell}$, the entries of the matrices $\left(g^{\alpha\beta}(t)\right), \left(\Gamma_\gamma^{\alpha\beta}(t)\right)$ are quasi-homogeneous polynomials of $t^1, \dots, t^\ell, \frac{1}{t^{\ell-m}}, \frac{1}{t^\ell}$ of degrees $d_\alpha+d_\beta$ and $d_\alpha+d_\beta-d_\gamma$. 
\end{cor}
\subsection{The generalized Frobenius manifold  structures} \label{zh-02-09-4}
For each fixed $0\le m\le \ell$, from Theorems \ref{main}, \ref{pencil generator of C-type} it follows that there is a generalized Frobenius manifold structure of charge $d=1$ on $\mathcal{M}_D({C_\ell},\omega_1)$ defined by the flat pencil of metrics $g, \eta$ given in \eqref{pencil generator of C-type}. The unit vector field and the Euler vector field are given by
  \begin{align}\label{the unity for A-type}
e=\eta^{\sharp}(\omega_e)=-\eta^\sharp(\mathrm{d}\log y^1),\quad  E=\sum_{\alpha=1}^{\ell}d_\alpha t^\alpha\tangentvector{t^\alpha}, 
    \end{align}
with $d_\al$ defined by \eqref{zh-07-20-3} and \eqref{zh-07-20-4}. The structure constants of the generalized Frobenius manifold are polynomials in $t^1, \dots, t^\ell, \frac{1}{t^{\ell-m}}, \frac{1}{t^\ell}$. We can determine the potential of the generalized Frobenius manifolds by the relation
\[\frac{\p^2 F}{\p t^\al\p t^\beta}(t)=\frac1{2-d_\al-d_\beta}\,\eta_{\al\xi}\eta_{\beta\zeta}g^{\xi\zeta}(t),\quad \al,\beta=1,\dots,\ell.\]
\begin{rem}\label{01-18-e}
From the above construction, we see that the generalized Frobenius manifold structures on $\mathcal{M}_{\ell,m}$ and $\mathcal{M}_{\ell,\ell-m}$ are equivalent.
\end{rem}

\subsection{Examples}
In this subsection, we give some examples to illustrate the above constructure of generalized Frobenius manifold structures associated to $(C_\ell, \omega_1)$. 

\begin{ex}
 Let $(R, \omega)=(C_2,\omega_1)$. We have the $W_a(R)$-invariant $\lambda$-Fourier polynomials
 \begin{align*}
 y^1&=e^{2\pi i x^1}+\lm e^{2\pi i(x^1-x^2)}+\lm e^{-2\pi i(x^1-x^2)}+\lm^2 e^{-2\pi i x^1},\\
 y^2&=e^{2\pi i x^2}+e^{2\pi i(2 x^1-x^2)}+\lm^2 e^{-2\pi i x^2}+\lm^2 e^{-2\pi i(2 x^1-x^2)}.
 \end{align*}
The contravariant metric on $V\otimes\mathbb{C}$ is given by
		\[\left((\dif x^i, \dif x^j)\right)=\left((\alpha^\vee_i, \alpha^\vee_j)\right)^{-1}=\begin{pmatrix}1&1\\ 1 &2\end{pmatrix}, \]
which induces the metric 
\[\left(g_\lambda^{ij}(y)\right)=\begin{pmatrix}
    -(y^1)^2+2\lambda y^2+8\lambda^2&-y^1y^2+4\lambda y^1\\[3 pt]
    -y^1y^2+4\lambda y^1&-2(y^2)^2+4(y^1)^2-8\lambda y^2
\end{pmatrix}.\]
From Theorem \ref{pencil generator of C-type} we know that we have three different choices pencil generators by taking $m=0, 1, 2$, and $g=(g^{ij})=(g_0^{ij})$ is independent of the choice of pencil generators. Note that the generalized Frobenius structure for $m=0, 2$ are isomorphic, so we only need to consider the case of $m=0$ and $m=1$. 
\vskip 0.1 truecm 
\item[\bf{Case 1.}] By taking $m=0$, we get the pencil generators
    \begin{align*}
z^1=y^1+4\lambda,\quad z^2=y^2-4\lambda,
    \end{align*}
and the metric
        \[\eta=(\eta^{ij}(z))=\begin{pmatrix}
    8z^1+2z^2& 4z^2\\
    4z^2 & -32z^1-24z^2
\end{pmatrix}. \]
The variables $\tau^\al$, $w^\al$, $v^\al$ that are introduced in the last subsection satisfy the relations
\begin{align*}
\tau^1&=z^1,\quad \tau^2=2 z^1+z^2;\\
w^1&=\tau^1-\frac16 \tau^2=v^1 v^2,\quad w^2=\tau^2=(v^2)^4.
\end{align*}
The flat coordinates are given by
    \[t^1=v^1,\quad t^2=v^2.\]
In these flat coordinate the flat pencil of metrics has the form
\begin{align*}
\left(\eta^{\alpha\beta}(t)\right)&=\begin{pmatrix}0&2\\2&0\end{pmatrix}, \\[4pt]\left(g^{\alpha\beta}(t)\right)&=
\begin{pmatrix}
    \frac{-(t^2)^9+9 t^1 (t^2)^6-27 (t^1)^2 (t^2)^3+27 (t^1)^3}{108 (t^2)^3} & \frac{(t^2)^6-24 t^1 (t^2)^3-18 (t^1)^2}{72 (t^2)^2} \\[8pt]
 \frac{(t^2)^6-24 t^1 (t^2)^3-18 (t^1)^2}{72 (t^2)^2} & \frac{3 t^1-(t^2)^3}{12 t^2} 
\end{pmatrix}.
 \end{align*}
We have the potential 
\[F=\frac{1}{48}\frac{(t^1)^3}{t^2}-\frac{1}{48}(t^1)^2(t^2)^2+\frac{1}{1440}t^1(t^2)^5-\frac{1}{36288}(t^2)^8\]
of the generalized Frobenius manifold.
The Euler vector field and the unity are given by 
\begin{align*}
    E=&\frac{3}{4}t^1\tangentvector{t^1}+\frac{1}{4}t^2\tangentvector{t^2}, \\
             e=&-\frac{4}{6t^1t^2+(t^2)^4}\left((3t^1+2(t^2)^3)\tangentvector{t^1}+3t^2\tangentvector{t^2}\right). 
    \end{align*}
\vskip 0.1truecm
\item[\bf{Case 2.}] By taking $m=1$, 
we have the pencil generators
    \begin{align*}
z^1=y^1, \quad z^2=y^2+4\lambda, 
    \end{align*}
and the metric
\[\eta=\left(\eta^{ij}(y)\right)=\begin{pmatrix}
    2z^2& 8z^1\\
    8z^1 & 8z^2
\end{pmatrix}.
\]
The variables $\tau^\al$, $w^\al$, $v^\al$ that are introduced in the last subsection satisfy the relations
\begin{align*}
\tau^1&=\frac12 z^1+\frac14 z^2,\quad \tau^2=-\frac12 z^1+\frac14 z^2;\\[3pt]
w^1&=\tau^1=(v^1)^2,\quad w^2=\tau^2=(v^2)^2.
\end{align*}
The flat coordinates are given by
\[t^1=v^1,\quad t^2=v^2.\]
In these flat coordinate the flat pencil of metrics has the form
\[\left(\eta^{\al\beta}\right)=\begin{pmatrix}1 & 0\\ 0 &1\end{pmatrix},\quad
\left(g^{ij}(t)\right)=\begin{pmatrix}
    -\frac{1}{4}((t^1)^2+(t^2)^2)&-\frac{1}{2}t^1t^2\\[3pt]
    -\frac{1}{2}t^1t^2&  -\frac{1}{4}((t^1)^2+(t^2)^2)
\end{pmatrix}. \]
The potential of the generalized Frobenius manifold has the form
\[F=-\frac{1}{48}(t^1)^4-\frac18 (t^1)^2(t^2)^2-\frac1{48}(t^2)^4, \]
and the Euler vector field  and the unity are given by
\[E=\frac{1}{2}t^1\frac{\partial}{\partial t^1}+\frac{1}{2}t^2\frac{\partial}{\partial t^2}, \quad
e=-\frac{2}{(t^1)^2-(t^2)^2}\left(t^1\frac{\partial}{\partial t^1}-t^2\frac{\partial}{\partial t^2}\right). \]
\end{ex}

\begin{ex} 
Let $(R, \omega)=(C_3,\omega_1)$. We have the $W_a(R)$-invariant $\lambda$-Fourier polynomials
    \begin{align*}
y^1=&\,e^{2\pi i x^1}+\lm e^{2\pi i(x^1-x^2)}+\lm e^{-2\pi i(x^1-x^2)}+\lm e^{2\pi i(x^2-x^3)}+\lm e^{-2\pi i(x^2-x^3)}+\lm^2 e^{-2\pi i x^1},\\
y^2=&\,e^{2\pi i x^2}+e^{2\pi i (2 x^1-x^2)}+e^{2\pi i(x^1+x^2-x^3)}+e^{2\pi i(x^1-x^2+x^3)}\\
&+\lm e^{2\pi i(x^1-x^3)}
+\lm e^{-2\pi i(x^1-2x^2+x^3)}+
\lm e^{-2\pi i(x^1-x^3)}+\lm e^{2\pi i(x^1-2x^2+x^3)}\\
&+\lm^2 e^{-2\pi i x^2}+
\lm^2 e^{-2\pi i(2x^1-x^2)}
+\lm^2 e^{-2\pi i(x^1-x^2+x^3)}+\lm^2 e^{-2\pi i(x^1+x^2-x^3)},\\
y^3=&\,e^{2\pi i x^3}+e^{2\pi i(2x^1-x^3)}+e^{2\pi i(2x^2-x^3)}+e^{2\pi i(2x^1-2x^2+x^3)}+\lm^2 e^{-2\pi i x^3}\\
&+\lm^2 e^{-2\pi i(2x^1-2x^2+x^3)}+
\lm^2 e^{-2\pi i(2x^1-x^3)}+\lm^2
e^{-2\pi i(2x^2-x^3)}.
    \end{align*}
The contravariant metric on $V\otimes\mathbb{C}$ is given by
\[\left((\dif x^i, \dif x^j)\right)=\left((\alpha^\vee_i, \alpha^\vee_j)\right)^{-1}=\begin{pmatrix}1&1&1\\ 1 &2&2\\1&2&3\end{pmatrix},\]
which induces the metric 
\[(g_\lambda^{ij}(y))=\begin{pmatrix}
    -(y^1)^2+2\lambda y^2+12\lambda^2&-y^1y^2+8\lambda y^1+3\lambda y^3&-y^1y^3+4\lambda y^2\\[3pt]
    -y^1y^2+8\lambda y^1+3\lambda y^3&-2(y^2)^2+8(y^1)^2+2y^1y^3-8\lambda y^2&-2y^2y^3+4y^1y^2-12\lambda y^3\\[3pt]
    -y^1y^3+4\lambda y^2&-2y^2y^3+4y^1y^2-12\lambda y^3&-3(y^3)^2+4(y^2)^2-8y^1y^3
\end{pmatrix}. \]
From Theorem \ref{pencil generator of C-type} we know that we have four different choices of pencil generators by take $m=0, 1, 2, 3$, and $g=(g^{ij})=(g_0^{ij})$ is independent of the choice of pencil generators. Note that the generalized Frobenius manifold structure for $m=0$ is isomorphic to the one for $m=3$, and the generalized Frobenius manifold structure for $m=1$ is isomorphic to the one for $m=2$, so we only need to consider the case of $m=0$ and of $m=1$. 
\vskip 0.1truecm
\item[\textbf{Case 1}.] By taking $m=0$ 
 we get the pencil generators
    \begin{align*}
z^1=y^1+6\lambda,\quad z^2=y^2-12\lambda, \quad z^3=y^3+8\lambda, 
    \end{align*}
and the metric
        \[\eta=\left(\eta^{ij}(z)\right)=\begin{pmatrix}
    12z^1+2z^2& -4z^1+6z^2+3z^3&8z^1+4z^2+6z^3\\[3pt]
    -4z^1+6z^2+3z^3 & -112z^1-56z^2-12z^3 &48z^1-8z^2-36z^3\\[3pt] 8z^1+4z^2+6z^3&48z^1-8z^2-36z^3&64z^1+96z^2+96z^3
\end{pmatrix}. \]
The variables $\tau^\al$, $w^\al$, $v^\al$ satisfy the relations
\begin{align*}
\tau^1&=z^1,\quad \tau^2=4 z^1+z^2,\quad \tau^3=4 z^1+2 z^2+z^3;\\[3pt]
w^1&=\tau^1-\frac16\tau^2+\frac1{30}\tau^3=v^1 v^3,\quad w^2=\tau^2-\frac14\tau^3=v^2 (v^3)^4,\quad w^3=\tau^3=(v^3)^6.
\end{align*}
The flat coordinates satisfy the relations
\[t_1=v^1-\frac1{12} (v^2)^2 v^3,\quad t_2=v^2 v^3,\quad t^3=v^3.\]
In these flat coordinate the flat metric has the form
\[(\eta^{\alpha\beta}(t))=\begin{pmatrix}
    0&0&2\\0&12&0\\2&0&0
\end{pmatrix}. \]
The potential of the generalized Frobenius manifold has the expression
\begin{align*}
    F=&\frac{1}{24}\frac{(t^1)^2t^2}{t^3}-\frac{1}{48} (t^1)^2 (t^3)^2-\frac{1}{216}\frac{t^1 (t^2)^3}{ (t^3)^2}-\frac{1}{288} t^1 (t^2)^2 t^3\\&+\frac{1}{1440}t^1t^2(t^3)^4-\frac{1}{60480}t^1 (t^3)^7+\frac{1}{4320 }\frac{(t^2)^5}{(t^3)^3}-\frac{1}{6912}(t^2)^4\\&+\frac{1}{17280}(t^2)^3 (t^3)^3-\frac{1}{34560}(t^2)^2 (t^3)^6+\frac{1}{345600}t^2(t^3)^9-\frac{1}{7603200}(t^3)^{12},
\end{align*}
and the Euler vector field and the unity are given by
\begin{align*}
E=&\,\frac{5}{6}t^1\tangentvector{t^1}+\frac{1}{2}t^2\tangentvector{t^2}+\frac{1}{6}t^1\tangentvector{t^1}, \\
    e=&-\frac{120}{10(t^2)^2+120t^1t^3+20t^2(t^3)^3+(t^3)^6}\left(\left(2t^1+t^2t^3+\frac{1}{10}(t^3)^5\right)\tangentvector{t^1}\right.
   \\
   &\left.+2\left(t^2+(t^3)^3\right)\tangentvector{t^2}+2t^3\tangentvector{t^3}\right). 
   \end{align*}
   
\item[\textbf{Case 2.}] By taking $m=1$ we
get the pencil generators
\[ z^1=y^1+2\lambda, \quad z^2=y^2+4\lambda, \quad z^3=y^3-8\lambda, \] and
and the metric
\[\eta=\left(\eta^{ij}(z)\right)=\begin{pmatrix}
    4z^1+2z^2&12z^1+2z^2+3z^3&-8z^1+4z^2+2z^3\\[1pt]
    12z^1+2z^2+3z^3&-16z^1+8z^2-4z^3&-16z^1-24z^2-4z^3\\[1pt]
    -8z^1+4z^2+2z^3&-16z^1-24z^2-4z^3&-64z^1-32z^2-32z^3
\end{pmatrix}. \]
The variables $\tau^\al$, $w^\al$, $v^\al$ satisfy the relations
\begin{align*}
\tau^1&=\frac34 z^1+\frac18 z^2-\frac1{16} z^3,\quad \tau^2=z^1+\frac12 z^2+\frac14 z^3,\quad \tau^3=-\frac14 z^1+\frac18 z^2-\frac1{16}z^3;\\[3pt]
w^1&=\tau^1-\frac16\tau^2=v^1 v^2,\quad w^2=\tau^2=(v^2)^4,\quad w^3=\tau^3=(v^3)^2.
\end{align*}
The flat coordinates are given by
\[t_1=v^1,\quad t_2=v^2,\quad t^3=v^3.\]
In these flat coordinate the flat metric has the form
\[(\eta^{\alpha\beta}(t))=\begin{pmatrix}
    0&2&0\\2&0&0\\0&0&1
\end{pmatrix}.\]
The potential of the generalized Frobenius manifold has the expression
\begin{align*}
    F=&\,\frac{1}{48}\frac{(t^1)^3}{t^2}-\frac{1}{48} (t^1)^2 (t^2)^2+\frac{1}{1440}t^1 (t^2)^5-\frac{1}{36288}(t^2)^8\\&-\frac{1}{8} t^1 t^2(t^3)^2+\frac{1}{96}  (t^2)^4(t^3)^2 -\frac{1}{48}(t^3)^4,
\end{align*}
and the Euler vector field and the unity are given by
\begin{align*}
E&=\frac{3}{4}t^1\tangentvector{t^1}+\frac{1}{4}t^2\tangentvector{t^2}+\frac{1}{2}t^3\tangentvector{t^3}, \\
e&=-\frac{12}{6t^1t^2+(t^2)^4-6(t^3)^2}\left(\left(t^1+\frac{2}{3}(t^2)^3\right)\tangentvector{t^1}+t^2\tangentvector{t^2}-t^3\tangentvector{t^3}\right). 
\end{align*}
\end{ex}

\subsection{Landau-Ginzburg superpotential}
As for the $\mathcal{M}_D(A_\ell,\omega_\ell)$ case, in this subsection we are to represent the generalized Frobenius manifold structures on $\mathcal{M}_D(C_\ell,\omega_1)$ in terms of superpotentials.

Consider the following rational functions of $p$:
\begin{equation}\label{01-18-b}
   \Lambda(p)=\frac{p^2-1}{p^{2m}}\left(\sum_{j=1}^\ell a_j p^{2(\ell-j)}\right),\quad a_1,\dots,a_\ell\in\mathbb{C},\,m=0,\dots,\ell.
\end{equation}
Let $\widetilde{\mathcal{M}}_{\ell,m}$ be the space
\[\widetilde{\mathcal{M}}_{\ell,m}=\left\{(a_1,\dots,a_\ell)\in\mathbb{C}^\ell\right\}.\]
We define the following two tensors on $\widetilde{\mathcal{M}}_{\ell,m}$ as follows:
\begin{align}\label{01-13-1}
    \tilde{\eta}(\p',\p'')&=\sum_{q:\,\Lambda'(q)=0}\mathop{\mathrm{Res}}_{p=q}\frac{\p'(\Lambda)\p''(\Lambda)}{\Lambda'(p)}\frac{\dd p}{(p^2-1)^2},\\   
     \tilde{c}(\p',\p'',\p''')&=\sum_{q:\,\Lambda'(q)=0}\mathop{\mathrm{Res}}_{p=q}\frac{\p'(\Lambda)\p''(\Lambda)\p'''(\Lambda)}{\Lambda'(p)}\frac{\dd p}{(p^2-1)^2},\label{zh-02-09-1}
\end{align}
where the summation runs over the critical points of $\Lambda(p)$, including the critical point at infinity for the $m=\ell$ case.

\begin{lem}
    For any given $m=0,\dots,\ell$, the tensor $\tilde{\eta}$ is a flat metric on $\widetilde{\mathcal{M}}_{\ell,m}$, and it has flat coordinates
  \begin{align}
&\tilde{t}^\al=\frac{1}{2(\ell-m-\al)+1}\mathop{\mathrm{Res}}_{p=\infty}\left(\Lambda^{\frac{2(\ell-m-\al)+1}{2(\ell-m)}}(p)\frac{\dd p}{p^2-1}\right),\quad \al=1,\dots,\ell-m;\label{01-13-3a}\\
        &\tilde{t}^\beta=\frac{1}{2(\ell-\beta)+1}\mathop{\mathrm{Res}}_{p=0}\left(\Lambda^{\frac{2(\ell-\beta)+1}{2m}}(p)\frac{\dd p}{p^2-1}\right),\quad \beta=\ell-m+1,\dots,\ell.\label{01-13-3b}
    \end{align}
\end{lem}
\begin{proof}
    Let $k_1,k_2$ be the roots of $\Lambda(p)$ which have the following expansions:
\begin{align}\label{01-13-2}
k_1&=\Lambda^{\frac{1}{2(\ell-m)}}=a_1^{\frac{1}{2(\ell-m)}}\left(p+\frac{a_2-a_1}{2(\ell-m)a_1}\frac1p+O\left(\frac{1}{p^2}\right)\right),\quad p\to\infty,\\
k_2&=\Lambda^{\frac{1}{2m}}=(-a_\ell)^{\frac{1}{2m}}\left(\frac1p+\frac{a_\ell-a_{\ell-1}}{2ma_\ell}p+O\left(p^2\right)\right),\quad p\to0.
    \end{align}
We assume that for $z\in\mathbb{C}$, $\mathrm{arg}z\in(-\pi,\pi]$, then by using \eqref{01-13-3a}, \eqref{01-13-3b}, we obtain
    \begin{align}
        \frac{1}{2}\log\frac{p-1}{p+1}=\begin{cases}
            \sum_{\alpha=1}^{\ell-m}\frac{\tilde{t}^{\alpha}}{k_1^{2(\ell-m-\alpha)+1}}+O\left(1/k_1^{2(\ell-m)}\right),&k_1\to\infty,\\
            \frac{\pi}{2}i+\sum_{\beta=\ell-m+1}^{\ell}\frac{\tilde{t}^{\beta}}{k_2^{2(\ell-\beta)+1}}+O\left(1/k_2^{2m}\right),&k_2\to\infty,
        \end{cases}
    \end{align}
    from which it follows that
\begin{align}
    &\frac{1}{p^2-1}\frac{\p p(k,\tilde{t})}{\p\tilde{t}^\al}=\begin{cases}
        \frac{1}{k_1^{2(\ell-m-\al)+1}}+O\left(1/k_1^{2(\ell-m)}\right),&k_1\to\infty,\\
        O\left(1/k_2^{2m}\right),&k_2\to\infty,
    \end{cases}\quad \mbox{for}\ \alpha=1,\dots,\ell-m,\label{01-18-1}\\
    &\frac{1}{p^2-1}\frac{\p p(k,\tilde{t})}{\p\tilde{t}^\beta}=\begin{cases}
       O\left(1/k_1^{2(\ell-m)}\right),\quad &k_1\to\infty,\\
        \frac{1}{k_2^{2(\ell-\beta)+1}}+O\left(1/k_2^{2m}\right),\quad &k_2\to\infty,
    \end{cases}\quad \mbox{for}\ \beta=\ell-m+1,\dots,\ell.\label{01-18-2}
\end{align}
By using the implicit function theorem one can obtain the following relation:
\[\p_{\tilde{t}^\al}(\Lambda(p(k,\tilde{t}))=-\Lambda'(p)\frac{\p p(k,\tilde{t})}{\p\tilde{t}^\al},\quad\dd\Lambda=2(\ell-m)k_1^{2(\ell-m)-1}\dd k_1=2mk_2^{2m-1}\dd k_2.\]
Thus, when $m\ne 0, \ell$ we have
\begin{align}
    &\tilde{\eta}(\p_{\tilde{t}^\al},\p_{\tilde{t}^\beta})=-\left(\mathop{\mathrm{Res}}_{p=\infty}+\mathop{\mathrm{Res}}_{p=0}\right)\frac{1}{(p^2-1)^2}\frac{\p p}{\p\tilde{t}^\al}\frac{\p p}{\p\tilde{t}^\beta}\dd\Lambda,\notag\\
    =&-\left(\mathop{\mathrm{Res}}_{k_1=\infty}\frac{2(\ell-m)}{(p^2-1)^2}\frac{\p p}{\p\tilde{t}^\al}\frac{\p p}{\p\tilde{t}^\beta}k_1^{2(\ell-m)-1}\dd k_1+\mathop{\mathrm{Res}}_{k_2=\infty}\frac{2m}{(p^2-1)^2}\frac{\p p}{\p\tilde{t}^\al}\frac{\p p}{\p\tilde{t}^\beta}k_2^{2m-1}\dd k_2\right).\label{zh-02-08}
\end{align}
By using \eqref{01-18-1}, \eqref{01-18-2}, we arrive at
\[(\tilde{\eta}_{\al\beta})=(\tilde{\eta}(\p_{\tilde{t}^\al},\p_{\tilde{t}^\beta}))=\begin{pmatrix}
    \tilde{A}_1& \\
     &\tilde{A}_2
\end{pmatrix},\]
where $\tilde{A}_1,\tilde{A}_2$ are $(\ell-m)\times (\ell-m)$ and $m\times m$ matrices, which are anti-diagonal matrices with anti-diagonal elements $2(\ell-m)$ and $2m$ respectively. 

When $m=0$, the function $\Lambda(p)$ has a critical point at $p=0$; when $m=\ell$, it has a critical point at $p=\infty$. Thus, in the derivation of  $\tilde{\eta}(\p_{\tilde{t}^\al},\p_{\tilde{t}^\beta})$ as in \eqref{zh-02-08}, we only need to take residue at $p=\infty$ for the $m=0$ case, and to take residue at $p=0$ for the $m=\ell$ case. In these two cases the matrix $(\tilde{\eta}_{\al\beta})$ is anti-diagonal with anti-diagonal elements $2\ell$. The lemma is proved.
\end{proof}

We are to show that the operation of multiplication on the tangent spaces of $\widetilde{\mathcal{M}}_{\ell,m}$ defined by \eqref{zh-01-07-1} yields a generalized Frobenius manifold structure on $\widetilde{\mathcal{M}}_{\ell,m}$. Let $\pm q_1,\dots,\pm q_\ell$ be the distinct critical points of $\Lambda(p)$, where we assume that $q_i\neq\pm1$. Then for $m\not=\ell$ we have
\begin{equation}
\Lambda'(p)=\frac{2(\ell-m)a_1}{p^{2m+1}}\prod_{i=1}^\ell(p^2-q_i^2).
\label{01-28-0}
\end{equation}
When $m=0$, the function $\Lambda(p)$ has a critical point at $p=0$, which we denote by $q_\ell=0$. For the $m=\ell$ case, we can easily see, by making a change of variable $p\to 1/p$, that the formulae \eqref{01-13-1},
\eqref{zh-02-09-1} yield the same Frobenius manifold structure as for the $m=0$ case, so in what follows we assume that $m\ne \ell$.

\begin{lem} The following relations hold true:
\begin{equation}\label{01-28-1}
\Lambda''(q_i)=\Lambda''(-q_i)=\left.\frac{c_{i,m}p\Lambda'(p)}{p^2-q_i^2}\right|_{p=q_i},\end{equation}
where $c_{i,m}=2-\delta_{i,\ell}\delta_{m,0}$.
\end{lem}
\begin{proof}
From \eqref{01-28-0} it follows that
\[\Lambda''(p)=\Lambda'(p)\left(\sum_{i=1}^\ell\frac{2p}{p^2-q_i^2}-\frac{2m+1}{p}\right).\]
So we get
\begin{align*}
\Lambda''(\pm q_i)&=\begin{cases}
\left.\frac{2p\Lambda'(p)}{p^2-q_i^2}\right|_{p=\pm q_i},&i\not=\ell \text{ or } i=\ell,m\not=0\\
-\frac{\Lambda'(p)}{p}\Big|_{p=0},&i=\ell,m=0.
\end{cases}\\
&=\left.\frac{c_{i,m}p\Lambda'(p)}{p^2-q_i^2}\right|_{p=\pm q_i}.
\end{align*}
The lemma is proved.
\end{proof}
Let $u_i=\Lambda(q_i)=\Lambda(-q_i)$ be the critical values of $\Lambda(p)$, then we have $\p_{u_i}\Lambda(p)|_{p=\pm q_j}=\delta_{ij}$. By using the Lagrange interpolation formula, we get
\begin{equation}\label{01-13-5}
    \p_{u_i}\Lambda(p)=\frac{c_{i,m}(p^2-1)p\Lambda'(p)}{(q_i^2-1)\Lambda''(q_i)(p^2-q_i^2)},
\end{equation}
which implies that
\begin{align}
    &\tilde{\eta}(\p_{u_i},\p_{u_j})=\sum_{k=1}^\ell\mathop{\mathrm{Res}}_{p=\pm q_k}\frac{\p_{u_i}\Lambda(p)\cdot\p_{u_j}\Lambda(p)}{\Lambda'(p)}\frac{\dd p}{(p^2-1)^2}\notag\\
    =&\sum_{k=1}^\ell\mathop{\mathrm{Res}}_{p=\pm q_k}\frac{c_{i,m}^2p^2\Lambda'(p)}{(q_i^2-1)(q_j^2-1)\Lambda''(q_i)\Lambda''(q_j)(p^2-q_i^2)(p^2-q_j^2)}\dd p\notag\\
    =&\frac{c_{i,m}^2\delta_{ij}}{(q_i^2-1)^2(\Lambda''(q_i))^2}\mathop{\mathrm{Res}}_{p=\pm q_i}\frac{p^2\Lambda'(p)}{(p^2-q_i^2)^2}\dd p=\frac{c_{i,m}\delta_{ij}}{(q_i^2-1)^2\Lambda''(q_i)}.\label{zh-02-08-2}
\end{align}
Similarly, we have
\begin{equation}
\tilde{c}(\p_{u_i},\p_{u_j},\p_{u_k})=\frac{c_{i,m}\delta_{ij}\delta_{ik}}{(q_i^2-1)^2\Lambda''(q_i)}.\label{zh-02-08-3}
\end{equation}
From \eqref{zh-02-08-2} and \eqref{zh-02-08-3} it follows that
\[\p_{u_i}\cdot\p_{u_j}=\delta_{ij}\p_{u_i},\quad i, j=1,\dots,\ell.\]
Thus $u_1,\dots,u_\ell$ are canonical coordinates of the multiplication, and the unit vector field and the Euler vector field are given by
\[\tilde{e}=\sum_{i=1}^\ell\p_{u_i},\quad\tilde{E}=\sum_{i=1}^\ell u_i\p_{u_i}.\]

\begin{lem}The unit vector field and the Euler vector field can also be represented in the form
   \begin{align}
    &\tilde{e}=-2\,\mathrm{grad}_{\tilde{\eta}}\log(a_1+\dots+a_\ell),\label{01-13-6a}\\
    &\tilde{E}=\sum_{\al=1}^\ell a_\al\p_{a_\al}=\sum_{\al=1}^{\ell-m}\frac{2(\ell-m-\al)+1}{2(\ell-m)}\tilde{t}^\al\p_{\tilde{t}^\al}+\sum_{\beta=1}^{m}\frac{2(\ell-\beta)+1}{2m}\tilde{t}^{\ell-m+\beta}\p_{\tilde{t}^{\ell-m+\beta}},\label{01-13-6b}
\end{align}
\end{lem}
\begin{proof}
    By using the relation
    \[\p_{u_i}\Lambda(p)=\frac{p^2-1}{p^{2m}}\left(\sum_{\al=1}^\ell (\p_{u_i}a_\al) p^{2(\ell-\al)}\right),\] 
    we obtain from \eqref{01-13-5} that
    \[\p_{u_i}a_\al=\frac{2c_{i,m}\times(-1)^{\al-1}(\ell-m)a_1\sigma_{\al-1}(\widehat{q_i^2})}{(q_i^2-1)\Lambda''(q_i)}.\]
    It follows that
    \begin{align*}
        &\p_{u_i}(a_1+\dots+a_\ell)=\frac{-2c_{i,m}(\ell-m)a_1}{(q_i-1)^2\Lambda''(q_i)}\prod_{j=1}^\ell(1-q_j^2)\\
        =&\frac{-c_{i,m}}{(q_i-1)^2\Lambda''(q_i)}\sum_{\al=0}^\ell2(\ell-m-\al)(a_{\al+1}-a_\al)\\
        =&\frac{-2c_{i,m}}{(q_i-1)^2\Lambda''(q_i)}(a_1+\dots+a_\ell),
    \end{align*}
    where we set $a_{\ell+1}=a_0=0$, thus \eqref{01-13-6a} is prove. The validity of \eqref{01-13-6b} comes from the homogeneity of $u_i$ and $\tilde{t}^\al$.
The lemma is proved.
\end{proof}

From Lemma \ref{25-12-17-1} we know that the 4-tensor $\tilde{c}_{\al\beta\gamma\xi}:=\partial_{\tilde{t}^\xi} \tilde{c}(\p_{\tilde{t}^\al},\p_{\tilde{t}^\beta},\p_{\tilde{t}^\gamma})$ is symmetric. We can also verify the homogeneity conditions (\ref{condition-3-1}) and (\ref{condition-3-2}), so $\mathcal M_{\ell,m}$ is a generalized Frobenius manifold with charge $d=1$. The intersection form of $\widetilde{\mathcal{M}}_{\ell,m}$ can be represented by
\begin{equation}
    \tilde{g}(\p',\p'')=\sum_{q:\,\Lambda'(q)=0}\mathop{\mathrm{Res}}_{p=q}\frac{\p'(\Lambda)\cdot\p''(\Lambda)}{\Lambda(p)\Lambda'(p)}\frac{\dd p}{(p^2-1)^2},
\end{equation}
which is defined on $\widetilde{\mathcal{M}}_{\ell,m}\setminus\Sigma_m$, where
\[\Sigma_m:=\{p\mid \Lambda(p)=0,\Lambda'(p)=0\}.\]
In terms of the canonical coordinates, we have
\[\tilde{g}(\dd u_i,\dd u_j)=\frac{u_i(q_i^2-1)^2\Lambda''(q_i)\delta_{ij}}{c_{i,m}}.\]

In what follows, we establish an isomorphism between the generalized Frobenius manifold structures defined on $\widetilde{\mathcal{M}}_{\ell,m}$ and $\mathcal{M}_{\ell,m}$, which is introduced in \eqref{01-18-a}.

\begin{thm}
   Let the map
   \[h\colon\widetilde{\mathcal{M}}_{\ell,m}\to\mathcal M_{\ell,m},\quad (a_1,\dots,a_\ell)\mapsto (z^1,\dots,z^\ell)\]
    be defined by
    \[z^j=[s^{\ell-j}]\left(\sum_{k=1}^\ell a_k(s-2)^{\ell-k}(s+2)^{k-1}\right),\]
    where $[s^{k}](f(s))$ stands for the coefficient of $s^k$ in the polynomial $f(s)$. Then $h$ gives an isomorphism between the generalized Frobenius manifold structures defined on $\widetilde{\mathcal{M}}_{\ell,m}$ and on $\mathcal{M}_{\ell,m}$.
\end{thm}
\begin{proof}
Let $s=2\frac{1+p^2}{1-p^2}$, then $\Lambda(p)$ given in \eqref{01-18-b} becomes
\begin{equation}
    \widetilde{\Lambda}(s)=\Lambda(p(s))=\frac{\sum_{j=1}^\ell z^j s^{\ell-j}}{(s-2)^m(s+2)^{\ell-m}}=\frac{Q(s)}{P_2(s)},
\end{equation}
where $Q(s)$ is given in \eqref{01-18-d}, and $P_2(s)$ is just $P_1(s)$ with $m$ replaced by $\ell-m$, see \eqref{01-18-c} and Remark \ref{01-18-e}.

Denote by 
\[s_i=2\frac{1+q_i^2}{1-q_i^2}, i=1,\dots,\ell\] 
the critical points of $\widetilde{\Lambda}(s)$, then for $m=1,\dots,\ell-1$ we have
\begin{align*}
    \frac{\dd^2\widetilde\Lambda}{\dd s^2}(s_i)=\frac{\dd^2\widetilde\Lambda}{\dd p^2}(q_i)\frac{4}{(s_i-2)(s_i+2)^3},\quad
    \p_{u_i}\widetilde\Lambda(s)=\frac{(s^2-4)\frac{\dd\widetilde\Lambda}{\dd s}(s)}{(s_i^2-4)(s-s_i)\frac{\dd^2\widetilde\Lambda}{\dd s^2}(s_i)}.
\end{align*}

We can calculate the generating function of $\tilde{\eta}(\dd z^j,\dd z^k)$ as follows:
\begin{align*}
    &\sum_{j,k=1}^\ell\tilde{\eta}(\dd z^j,\dd z^k)r^{\ell-j}s^{\ell-k}=\sum_{i,j=1}^\ell\tilde{\eta}(\dd u_i,\dd u_j)\p_{u_i}\widetilde\Lambda(r)\p_{u_j}\widetilde\Lambda(s)P_2(r)P_2(s)\\
    =&\sum_{i=1}^\ell\frac{2(r^2-4)(s^2-4)\frac{\dd\widetilde\Lambda}{\dd s}(r)\frac{\dd\widetilde\Lambda}{\dd s}(s)}{(s_i^2-4)\frac{\dd^2\widetilde\Lambda}{\dd s^2}(s_i)(r-s_i)(s-s_i)}P_2(r)P_2(s)\\
    =&\sum_{i=1}^\ell\frac{2(r^2-4)(s^2-4)\frac{\dd\widetilde\Lambda}{\dd s}(r)\frac{\dd\widetilde\Lambda}{\dd s}(s)}{(s_i^2-4)\frac{\dd^2\widetilde\Lambda}{\dd s^2}(s_i)}P_2(r)P_2(s)\left(\frac{1}{r-s_i}-\frac{1}{s-s_i}\right)\frac{1}{s-r},
\end{align*}
    here we employ a similar argument used in the proof of Lemma \ref{1 term}, since $(s^2-4)\frac{\dd\Lambda}{\dd s}(s)P_2(s)$ is a polynomial. Furthermore, we have
    \begin{align}
        &\sum_{j,k=1}^\ell\tilde{\eta}(\dd z^j,\dd z^k)r^{\ell-j}s^{\ell-k}=\frac{2((r^2-4)\frac{\dd\widetilde\Lambda}{\dd s}(r)-(s^2-4)\frac{\dd\widetilde\Lambda}{\dd s}(s))}{r-s}P_2(r)P_2(s)\notag\\
        =&\,2\left(\frac{s^2-4}{s-r}\left(Q'(s)P_2(r)-Q(s)\frac{P_2'(s)}{P_2(s)}P_2(r)\right)-\frac{r^2-4}{s-r}\left(Q'(r)P_2(s)-Q(r)\frac{P_2'(r)}{P_2(r)}P_2(s)\right)\right)\notag\\
        =&\,2\Big(\frac{r^2-4}{r-s}(Q'(r)P_2(s)+Q(r)P_2'(s))-\frac{s^2-4}{r-s}(Q'(s)P_2(r)+Q(s)P_2'(r))\notag\\
        &-\ell(Q(r)P_2(s)+Q(s)P_2(r))\Big).\label{zh-02-09-2}
    \end{align}
    Similarly, we have
\begin{equation}
\sum_{j,k=1}^\ell\tilde{g}(\dd z^j,\dd z^k)r^{\ell-j}s^{\ell-k}=-2\ell Q(r)Q(s)+2\frac{r^2-4}{r-s}Q'(r)Q(s)-2\frac{s^2-4}{r-s}Q(s)Q'(s).\label{zh-02-09-3}
\end{equation}
By comparing the formulae \eqref{zh-02-09-2}, \eqref{zh-02-09-3}
with \eqref{zh-07-19-1} and \eqref{generating of g lambda ij} respectively, we obtain
\[h^*\eta=\frac{1}{2}\tilde{\eta},\quad h^*g_0=\frac{1}{2}\tilde{g}.\]
It is also easy to prove that
\[h_*\tilde{e}=e,\quad h_*\tilde{E}=E.\]

For the $m=0$ case, note that $s_\ell=2$ is not a critical point of $\widetilde{\Lambda}(s)$, so we have
\[\left.\frac{\dd^2\widetilde{\Lambda}}{\dd p^2}(p)\right|_{p=0}=\left.8\frac{\dd\widetilde{\Lambda}}{\dd s}(s)\right|_{s=2},\quad\p_{u_\ell}\widetilde{\Lambda}(s)=\frac{(s+2)\frac{\dd\widetilde{\Lambda}}{\dd s}(s)}{\left.4\frac{\dd^2\widetilde{\Lambda}}{\dd s^2}(s)\right|_{s=2}}.\]
Thus the generating functions of $\tilde{\eta}(\dd z^j,\dd z^k)$ and $\tilde{g}(\dd z^j,\dd z^k)$ can be calculated similarly as above, and we can also reach the conclusion of the theorem. The theorem is proved.
\end{proof}

\begin{rem}
By using the change of variable $p\to 1/p$ in the formulae \eqref{01-13-1} and \eqref{zh-02-09-1}, we see that the generalized Frobenius manifolds $\widetilde{\mathcal{M}}_{\ell,m}$ and $\widetilde{\mathcal{M}}_{\ell,\ell-m}$ are equivalent.
\end{rem}

\begin{rem}
Let us rescale the flat coordinates $\tilde t^1,\dots, \tilde t^\ell$ to introduce the new flat coordinates
\begin{align*}
t^\al&=2^{\frac{6(\ell-m)+2\al-1}{4(\ell-m)}}(\ell-m)\tilde{t}^\al,\quad \al=1,\dots,\ell-m-1,\\
t^\beta&=2^{\frac{6(\ell-\beta)+4m+3}{4 m}}m\tilde{t}^\beta,\quad \beta=\ell-m+1,\dots,\ell-1,\\
t^\ell&=2^{\frac{3}{4m}} \tilde{t}^\ell,\quad 
t^{\ell-m}=2^{\frac{4(\ell-m)-1}{4(\ell-m)}}\tilde{t}^{\ell-m}.
\end{align*}
Then in these new flat coordinates, the potential $\tilde{F}(t)$ for the Frobenius manifold $\widetilde{\mathcal{M}}_{\ell,m}$ coincides with the potential $F(t)$ for the Frobenius manifold $\mathcal{M}_{\ell,m}$ which is constructed in Section \ref{zh-02-09-4}; the components of the flat metric $(\tilde{\eta}^{\al\beta})$ are related with that of the flat metric $(\eta^{\al\beta})$
given in Corollary \ref{zh-02-09-5} by
\[\left(\tilde{\eta}^{\al\beta}\right)=\left(\sqrt{2}\,\eta^{\al\beta}\right).\]
\end{rem}
    
\section{The Cases of $(B_\ell, \omega_1)$ and $(D_\ell, \omega_1)$}\label{B, D-type}
	\subsection{The Invariant $\lambda$-Fourier Polynomial Ring}
 Let $R$ be the root system of type $B_\ell$ or $D_\ell$ realized in the $\ell$-dimensional Euclidean space $V$ with orthonormal basis $e_1,\dots, e_\ell$.
Take the simple roots as follows:
\begin{align*}
&B_\ell\ \mbox{case:}\quad \al_1=e_1-e_2,\dots, \al_{\ell-1}=e_{\ell-1}-e_\ell,
\al_\ell=e_\ell;\\
&D_\ell\ \mbox{case:}\quad \al_1=e_1-e_2,\dots, \al_{\ell-1}=e_{\ell-1}-e_\ell,
\al_\ell=e_{\ell-1}+e_\ell.
\end{align*}
Then the fundamental weights are given by
\begin{align*}
\omega_i&=\al_1+2\al_2+\dots+(i-1)\al_{i-1}+i (\al_i+\al_{i+1}+\dots+\al_\ell),\ i=1,\dots, \ell-1,\\
\omega_\ell&=\frac12(\al_1+2\al_2+\dots+\ell\al_\ell)
\end{align*}
for the $B_\ell$ case, and 
\begin{align*}
\omega_i&=\al_1+2\al_2+\dots+(i-1)\al_{i-1}+i (\al_i+\dots+\al_{\ell-2})
+\frac12 i (\al_{\ell-1}+\al_\ell),\ i=1,\dots,\ell-2,\\
\omega_{\ell-1}&=\frac12\left(\al_1+2\al_2+\dots+(\ell-2)\al_{\ell-2}+\frac12\ell\al_{\ell-1}+\frac12(\ell-2)\al_\ell\right),\\
\omega_\ell&=\frac12\left(\al_1+2\al_2+\dots+(\ell-2)\al_{\ell-2}+\frac12(\ell-2)\al_{\ell-1}+\frac12\ell\al_\ell\right)
\end{align*}
for the $D_\ell$ case.

Take $\omega=\omega_1$, then $\kappa=1$, and the numbers $\theta_i=(\omega_i,\omega_1)$ are given by
\begin{align*}
&B_\ell\ \mbox{case:}\quad \theta_i=1,\ \theta_\ell=\frac12,\quad i=1,\dots,\ell-1;\\
&D_\ell\ \mbox{case:}\quad \theta_i=1,\ \theta_{\ell-2}=\theta_\ell=\frac12,\quad i=1,\dots,\ell-2.
\end{align*}
We define $\xi^1,\dots, \xi^{\ell}$ by the relation
\[ c\omega + x^1\alpha_1^\vee + \dots + x^\ell\alpha_\ell^\vee=\xi^1 e_1+\dots+\xi^\ell e_\ell,\]
then the basic generators of $\mathscr A^{W}$ are given by
\[
y^i=\lm\sigma_i(\zeta^1, \dots, \zeta^\ell), \ y^\ell=\lm^{\frac12}\sigma_\ell(\tilde\zeta^1, \dots, \tilde\zeta^\ell),\quad i=1, \dots, \ell-1\]
for the $B_\ell$ case, and by
\begin{align*}
 y^i&=\lm\sigma_i(\zeta^1, \dots, \zeta^\ell), \quad i =1,\dots,\ell-2,\\
 y^{\ell-1}&=\frac12\lm^{\frac12}\left(\sigma_\ell(\tilde\zeta^1_+, \dots, \tilde\zeta^\ell_+)+\sigma_\ell(\tilde\zeta^1_-, \dots, \tilde\zeta^\ell_-)\right),\\
 y^\ell&=\frac12\lm^{\frac12}\left(\sigma_\ell(\tilde\zeta^1_+, \dots, \tilde\zeta^\ell_+)-\sigma_\ell(\tilde\zeta^1_-, \dots, \tilde\zeta^\ell_-)\right)
 \end{align*}
for the $D_\ell$ case. Here we use the notations
\begin{align*}
    &\zeta^j=e^{2\pi i\xi^j}+e^{-2\pi i \xi^j},\quad  j=1, \dots, \ell, \\
    &\tilde\zeta^j=e^{\pi i\xi^j}+e^{-\pi i\xi^j},\quad  j=1, \dots, \ell,\\
    &\tilde\zeta^j_\pm=e^{\pi i\xi^j}\pm e^{-\pi i\xi^j},\quad  j=1, \dots, \ell.
\end{align*}

\subsection{The Relation between the case of $(B_\ell,\omega_1),(D_\ell,\omega_1)$ and that of $(C_\ell,\omega_1)$}
Motivated by \cite{JGP}, we perform the change of coordinates 
\begin{align}\label{B->C}
    y^j&\mapsto \hat{y}^j=y^j, \quad j=1, \dots, \ell-1, \nonumber\\
    y^\ell&\mapsto \hat{y}^\ell=(y^\ell)^2-\sum_{k=1}^{\ell-1}2^{\ell-k}y^k-2^\ell\lambda, 
\end{align}
in the $(B_\ell,\omega_1)$ case, and the change of coordinates 
\begin{align}\label{D->C}
    y^j&\mapsto \hat{y}^j=y^j, \quad j=1, \dots, \ell-2, \nonumber\\
    y^{\ell-1}&\mapsto \hat{y}^{\ell-1}=y^{\ell-1}y^\ell-\frac{1}{4}\sum_{k=1}^{\ell-2}\left(2^{\ell-k}-(-2)^{\ell-k}\right)y^k-\frac{1}{4}\left(2^\ell-(-2)^\ell\right)\lambda, \\
    y^\ell&\mapsto\hat{y}^\ell=(y^\ell)^2+(y^{\ell-1})^2-\frac{1}{2}\sum_{k=1}^{\ell-2}\left(2^{\ell-k}+(-2)^{\ell-k}\right) y^k-\frac{1}{2}\left(2^\ell+(-2)^\ell\right)\lambda, \nonumber
\end{align}
in the $(D_\ell,\omega_1)$ case, then in these new coordinates the metric $g_\lambda$ coincides  
with the one given by \eqref{generating of g lambda ij} for the $(C_\ell,\omega_1)$ case. Thus, the generalized Frobenius manifold structures that we obtain in this way from $(B_\ell, \omega_1)$ and $(D_\ell,\omega_1)$ are isomorphic to the one that we obtain from $(C_\ell, \omega_1)$.

	\section{Conclusions}
	\label{sec5}
Starting from an irreducible reduced root system $R$ in the Euclidean space $V$ and a fixed weight $\omega$, we introduce the $W_a(R)$-invariant $\lm$-Fourier polynomial ring $\mathcal{A}^W$ in \cite{JTZ2025-1}, and constructed a   generalized Frobenius manifold structure on the orbit space of the associated affine Weyl group $W_a(R)$ under the assumption of the existence of a set of so called pencil generators of $\mathcal{A}^W$. In this paper, we construct pencil generators for the root systems of type $A_\ell, B_\ell, C_\ell$ and $D_\ell$ with the choice $\omega_\ell$ for $A_\ell$ and $\omega=\omega_1$ for the other root systems.
We also show that in the $A_\ell$ case the structure constants of the associated Frobenius algebra are quasi-homogeneous polynomials in the flat coordinates of the flat metric $\eta$, and in other cases they are rational functions of the flat coordinates. We expect that this construction of generalized Frobenius manifolds also works for other choices of the weight $\omega$, and also for the exceptional root systems of type $G_2, F_4, E_6, E_7$ and $E_8$, and we will consider these cases in subsequent publications. It is also interested to study the possibility of constructing generalized Frobenius manifold structures on the orbit space of Jacobi groups \cite{Bertola, Bertola-1, Bertola-2}.
	
%
\vskip 0.3cm 
\noindent{\bf{Acknowledgement.}}
This work is supported by NSFC No.\,12571266.


	\medskip

	\noindent Lingrui Jiang, 
	
	\noindent Department of Mathematical Sciences,  Tsinghua University \\ 
	Beijing 100084,  P.R.~China\\
	jlr24@mails.tsinghua.edu.cn
	\medskip

    \noindent Si-Qi Liu, 
	
	\noindent Department of Mathematical Sciences,  Tsinghua University \\ 
	Beijing 100084,  P.R.~China\\
	liusq@tsinghua.edu.cn
	\medskip
    
	 \noindent Yingchao Tian, 
	
	\noindent Department of Mathematical Sciences,  Tsinghua University \\ 
	Beijing 100084,  P.R.~China\\
	tianyc23@mails.tsinghua.edu.cn
	\medskip
	
	\noindent Youjin Zhang, 
	
	\noindent Department of Mathematical Sciences,  Tsinghua University \\ 
	Beijing 100084,  P.R.~China\\
	youjin@tsinghua.edu.cn
\end{document}